\documentclass{amsart}

\usepackage{slashed, mathrsfs,  dutchcal, mathabx, esint}

\newcommand{\mc}{\mathcal}

\newcommand{\mb}{\mathbb}
\newcommand{\ul}{\underline}
\newcommand{\mr}{\mathring}
\newcommand{\wh}{\widehat}
\newcommand{\wc}{\widecheck}
\newcommand{\wt}{\widetilde}

\newcommand{\ol}{\overline}

\newcommand{\vtr}{\vartriangleright}

\title{Constant Mean Curvature surfaces with prescribed finite topologies}

 \usepackage{amsmath, amssymb, amsfonts, amsthm, amscd, graphicx, subfiles, xr, leftidx, slashed, enumitem, lipsum } \usepackage{mathrsfs}
 \usepackage{leftidx}
  \usepackage{xcolor}
  \usepackage{graphicx}

  \usepackage{verbatim}

\long\def\symbolfootnote[#1]#2{\begingroup%
\def\thefootnote{\fnsymbol{footnote}}\footnote[#1]{#2}\endgroup}

\theoremstyle{plain}
\numberwithin{equation}{section}

\newtheorem{theorem}{Theorem}[section]  
\newtheorem{lemma}[theorem]{Lemma}
\newtheorem{corollary}[theorem]{Corollary}
\newtheorem{proposition}[theorem]{Proposition}
\newtheorem{definition}[theorem]{Definition}
\newtheorem{remark}[theorem]{Remark}

\title{Complete embedded constant mean curvature surfaces in euclidean three space with prescribed topologies}
\author{Stephen J.  Kleene}
\address{Department of Mathematics, University of Rochester, Rochester, NY}
\email{skleene@ur.rochester.edu}

\begin{document}
\maketitle

\begin{abstract} In this  article, we construct complete embedded constant mean curvature surfaces in $\mb{R}^3$ with freely prescribed  genus and any  number of ends greater than or equal to four. Heuristically, the surfaces are obtained by resolving finitely many points of tangency between collections of spheres.  The construction relies a family of constant mean curvature surfaces constructed in \cite{Kleene}, constructed as graphs over catenoidal necks of small scale.
\end{abstract}

\section{Introduction}
\begin{figure} \label{CMCFigure}
\includegraphics[ width=.8\textwidth]{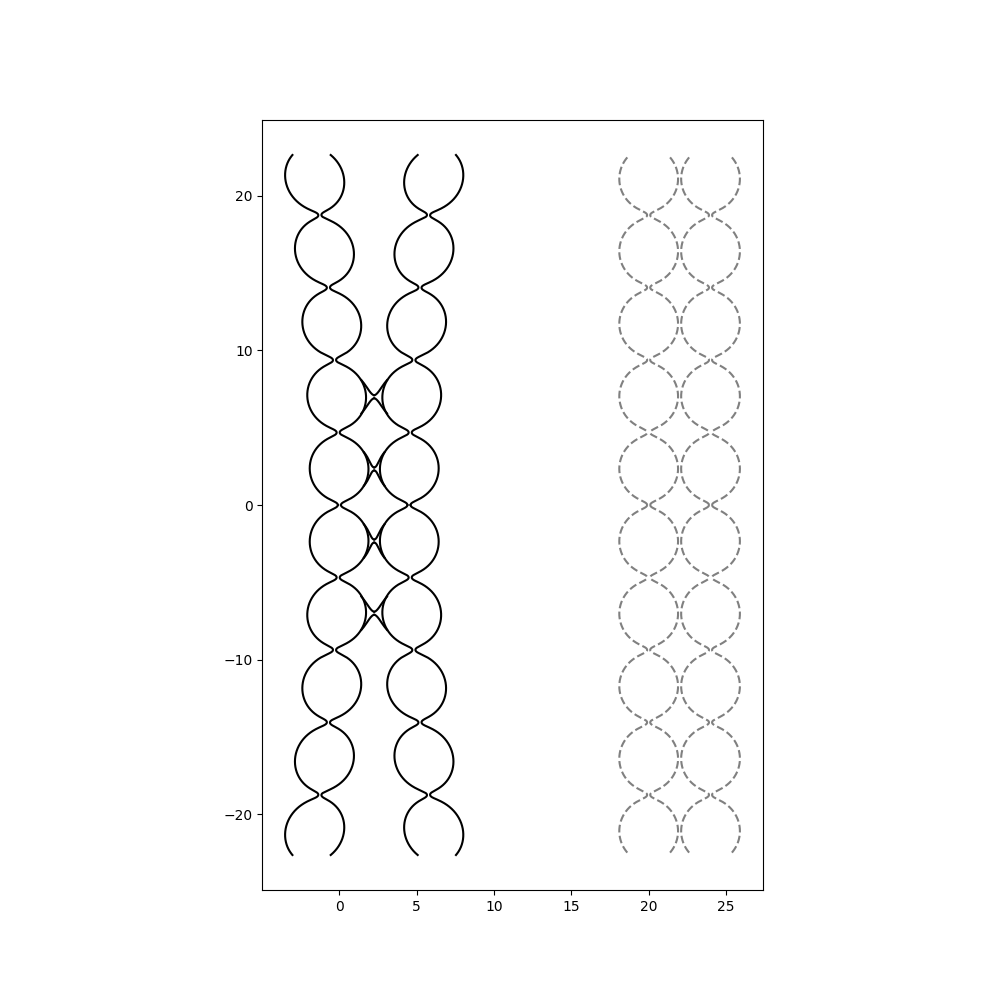}
\caption{A schematic depiction of one of surfaces constructed in this article, along with the limiting configuration comprising two touching Delaunay ends. This surface has  genus three and four ends, and is symmetric through the the $x$, $y$ and $z$ coordinate planes. Note the asymptotics of the 4 ends relative to the limiting configuration}
\end{figure}

Constant mean curvature (CMC) surfaces arise as critical points for the area functional under volume preserving deformations. They arise in physical world as soap films. In this article,  we construct families of CMC surfaces in euclidean three space with genus $g$ and $k$ ends, where $g$ and $k$ are arbitrary integers with $g \geq 0$ and $k \geq 4$.

\begin{theorem}\label{BigTheorem}
Given integers $k \geq 4$ and $g \geq 0$, there is a family  of complete embedded constant mean curvature surfaces in $\mb{R}^3$ with $k$ ends and genus $g$.
\end{theorem}

 The classical examples of constant mean curvarture surfaces in $\mb{R}^3$ with finite topology are round spheres, cylinders, and the family of rotationally invariant, singly periodic surfaces discovered by Delaunay in 1841 \cite{delaunay}. Hopf \cite{Hopf} showed that the only closed genus zero CMC surface in $\mb{R}^3$ is the round sphere, and Alexandrov \cite{Alexandrov} showed that the only closed embedded CMC surface is the round sphere. In contrast, Wente \cite{Wente} constructed genus zero closed immersed examples, at the time the only examples of complete CMC surfaces of finite topological type beyond the classical examples. The gluing techniques of Schoen \cite{Schoen} and Kapouleas \cite{Kapouleas},\cite{Kapouleas2}, \cite{Kapouleas3}, as well as those of Mazzeo  and Pacard \cite{mazzeo-pacard}, and Mazzeo-Pacard-Pollack \cite{mazzeo-pacard-pollack} led to proliferation of examples of complete CMC surfaces of finite topological type in the embedded, Alexandrov embedded, and immersed settings. In particular, Kapouleas in \cite{Kapouleas2} and \cite{Kapouleas3} constructed closed immersed CMC surfaces of arbitrary genus greater than two. Mazzeo and Pacard, \cite{mazzeo-pacard} and Mazzeo, Pacard and Pollack \cite{mazzeo-pacard-pollack} developed a general  connected sum construction for non-degenerate surfaces. In the embedded setting, the strongest existence results for finite topologies  are those of Breiner and Kapouleas in \cite{breiner-kapouleas1} and \cite{breiner-kapouleas2}, in which  complete embedded CMC surfaces with infinitely many topological types in euclidean spaces of dimension three and higher are constructed. More precisely, they construct embedded surfaces with prescribed numbers of ends, and for each fixed number of ends embedded surfaces of  finitely many distinct topological types are constructed.  In contrast, Meeks \cite{Meeks} showed that embedded CMC surfaces with finite topology have at least two ends, and Korevaar, Kusner,  and Solomon \cite{KKS} showed that each end of a finite topology CMC surface converges exponentially to a Delaunay end and, in the case two ends, it must be Delaunay. Theorem \ref{BigTheorem} is then optimal in the sense that we construct all possible finite topologies that can arise among embedded CMC surfaces with at least four ends. The topologies of the  Breiner-Kapouleas examples in  \cite{breiner-kapouleas1} and \cite{breiner-kapouleas2} are constrained by the fact that they are modeled on families of weighted  graphs, in which the nodes of the graph are replaced by spheres and edges are modeled on Delaunay ends with length approximately equal to an integer weight. The topology of the surface is then inherited from that of the underlying graph. The graphs must be carefully constructed in order to ensure embeddness of the resulting surfaces. In  $\mb{R}^3$, the surfaces constructed in \cite{breiner-kapouleas1} have genus $2k - 5$ where $k$ is the number of ends. In \cite{breiner-kapouleas2}, the authors construct hypersurfaces in higher dimensional spaces and remark that, due higher extra  space afforded by more dimensions, the graphs in \cite{breiner-kapouleas1} can be modified to obtain additional topological types for each $k$, although the specifics are not mentioned.  In contrast, the geometry of the surfaces we construct is comparatively simple and is easy to describe in terms of their limiting configuration.  In the case of four ends and maximal symmetry, the surfaces depend on a single parameter, and as the parameter tends to zero they converge to a singular limit comprising two tangent Delaunay ends differing by a translation.  The surfaces can thought of has having been obtained by removing finitely many of the points  of tangency and replacing them with catenoidal bridges. A basic modification of this idea introduces more ends into the construction. Since the ends of the limiting configuration are not separated, it is not immediately obvious that the surfaces we construct are in fact embedded, and  to establish this a careful treatment of the linearized problem is needed. Although we mention Delaunay ends above, they do not appear directly in our construction.  Rather, we construct our surfaces by resolving the points of tangency of configurations of tangentially touching spheres with catenoidal necks of small scale. Also, we remark that, although the surfaces we construct have a planar symmetry, no symmetry is required by our method. 
 
 \subsection{General comments on the proof}
The general  approach of our construction has its origins in the work  of Mazzeo-Pacard in \cite{mazzeo-pacard}. A novel feature of their approach, at least in the context of singular perturbations, is the use of the Dirichlet to Neumann operator in matching the Cauchy data of various component surfaces.  This is in contrast with the approach developed by  Kapouleas,  in which families of  globally defined  approximate solutions are constructed and the mapping properties of the jacobi operator are understood globally.  Both types of constructions are carried out subject to one or more parameters, which we refer to here as the \emph{scale} of the construction and denote by $\tau$, and which is assumed to small subject to various constraints. For small values of $\tau$, the relevant error term is  shown to be sufficiently small, and the linearized problem sufficiently regular, so that a version of Newton's Method can be employed to find exact solutions.  A general feature of these constructions is that the problem does not extend  $\tau = 0$,  which has to do with the fact that surfaces tend to singular limits as the scale $\tau$ tends to zero, and uniform $C^1$ estimates for the problem do not hold at small scales. Thus, perturbing away from $\tau = 0$ using methods requiring differentiability in $\tau$ are not available,  and instead arguments relying on continuity, and their attendant technical estimates, must be substituted.

The principal advance offered by our construction  is a formulation of the problem  that extends differentiably to $\tau = 0$. This allows for a direct use of the implict function theorem and abstract appeals to differentiabilty, and obviates the need for many of the technical estimates  that are required  when only continuity is available. This allows us to perturb the problem differentiably away from scale zero and obtain improved estimates for solutions  in terms of the scale. These improved estimates are fundamental to our construction, which relies on the ability to linearize in $\tau$ at $\tau = 0$, and would not be possible otherwise.  Our formulation relies on a  family of constant mean curvature necks constructed in \cite{Kleene}, whose relavant properties are discussed in detail in Section  \ref{CMCNecks}.

\subsection{Notation, terminology and conventions}
In this article, we mostly  follow generally established notation that is widely in use, and thus we will avoid commenting extensively on it. So, for example,  $\mb{R}^n$ denotes $n$-dimensional euclidean space and $e_1, \ldots e_n$ the associated standard basis. We will regard, unless otherwise indicated, $\mb{R}^{n}$ as canonically included in $\mb{R}^{n +1}$ under the identification with the plane $\{x_n = 0 \}$.  We use the non-standard notation $f \leq^C g$ to mean $f \leq C g$, which is convenient in the case that the right hand side contains many terms and would otherwise need to be enclosed in parentheses. We also use $f \cong g$ to mean that the quantities $f$ and $g$ are equal up to a fixed positive constant. We use $f < < g$ to mean that the ratio $f/g$ tends  $0$ for asymptotic values of $f$ and $g$.

We take the mean curvature of an oriented surface $S$ in $\mb{R}^3$ to be the sum of the principal curvatures. Thus, the sphere of radius $\frac{1}{2}$ in $\mb{R}^3$ has mean curvature $1$ when  oriented by the outward pointing unit normal.  The \emph{stability operator} or \emph{jacobi operator} of a surface is  given by $L_{S} = \Delta_S + |A_S|^2$, where here $|A_S|^2$ denotes the square length of the second fundamental form of $S$ and $\Delta_S$ the Laplace-Beltrami operator. If $X$ is a vector field on $S$, the linear change $\dot{H}$ at $t = 0$ in  the mean curvature $H_{t}$  of the one parameter family of surfaces $S_t : = S + t  X$  is given by
\begin{align}\label{MCVariation}
\dot{H} = L_S X^\perp + \nabla_{X^\top} H.
\end{align}
Above, $X^\perp : = X \cdot N$ denotes the normal part of $X$ along $S$, $X^\top  : = X - X^\perp N$ denotes the tangential part and $\nabla_{X^\top} H : = \nabla H \cdot X^\top$ denotes direction derivative of $H$ in the direction $X$ . A \emph{jacobi field} on $S$ is a function $u$ on $S$ satisfying $L_{S} u = 0$. For a constant mean curvature surface, the gradient of the mean curvature vanishes. Thus, since  the mean curvature is invariant under rigid motions, the normal part of the translational and rotational vector fields along $S$ are jacobi fields. We refer to jacobi fields of this form as \emph{geometric}. A geometric jacobi field generated by a translation is said to be \emph{translational} and one generated by a rotation is said to be \emph{rotational}. When $S$ is non-degenerate--meaning that the Dirichlet kernel of the stability operator of $S$ is trivial--we let $\dot{\mc{h}}_{S, f}$ denote the unique jacobi field on $S$ with trace $f$.

Throughout, we rely on the language of disjoint unions, which in places we use somewhat informally. If $I$ is an index set  $F = (F_i)_{i \in I}$ is a family of sets, then the \emph{disjoint union} of the family is given by $\coprod_{i \in I} F_i  = \bigcup_{i} \{i\} \times F_i$. Although this is standard, we record this precisely as \emph{disjoint union} is frequently used informally to mean the union of disjoint sets. In this article, we will frequently consider disjoint unions in which the index set is only implicitly mentioned, if at all, since the existence of an indexing is trivial and distinct indexings give rise to isomorphic objects.  We define the \emph{canonical projection} $\pi:  \coprod_{i \in I} F_i \rightarrow \bigcup_{i} F_i$ the mapping given by $(i, p) \in \{ i\} \times F_i \mapsto p$.  When the sets $F_i$ are subsets of $\mb{R}^3$--or more generally, belong to a common space--we  will sometimes  say that the canonical projection \emph{projects into $\mb{R}^3$}. When the sets $F_i$ are mutually disjoint, the disjoint union and the union are isomorphic under the canonical projection. When this is the case, we will not carefully distinguish between the two unless it is important to do so. 

A one parameter family of surfaces $S_t$ is smooth  if there is smooth  family $V_t$ of vector fields such that $S_t   = S_0 + V_t$. Let  $S_t$ be  a smooth one parameter family of oriented surfaces and assume that $S_0$ is smooth.  The \emph{ normal variation field} on $S_0$ generated by the family $S_t$ is given by $\dot{S}^\perp : = \left(\left.\frac{d}{dt} \right|_{t = 0} V_t \right) \cdot N_{S_0}$ where $N_{S_0}$ denotes the unit normal to the surface $S_0$. Suppose that $V_t$ and $W_t$ are smooth families of vector fields on $S_0$ satisfying $S_t = S_0 + V_t = S_0 + W_t$. Then there is a smooth family of a diffeomorphisms $\phi_t$ of $S_0$ such that $p + V_t (p) = \phi_t(p) + W_t(\phi_t(p))$ for all $p \in S_0$. Differentiating in $t$ at $t = 0$ gives
\[
\dot{V} = \dot{W} + \dot{\phi}
\]
and thus the normal parts $\dot{V}^\perp$ and $\dot{W}^\perp$ agree, and the normal variation field is well-defined. When $S_t$ is a smooth one parameter family of constant mean curvature surfaces, we have from (\ref{MCVariation}) that
\[
L_{S_0} \dot{S}^\perp = 0,
\]
and thus the normal variation field of a smooth one parameter family of CMC surfaces is a jacobi field on $S_0$. 
\subsection{Outline of the proof}
In Section \ref{CMCNecks}, we record the properties of the family of constant mean curvature necks  constructed in \cite{Kleene} that are relevant to our construction.  In Section \ref{PreEmSpheres}, we define the family of  surfaces  $\mc{C}$ that we will use to construct complete embedded CMC surfaces with prescribed topologies and describe the parameters of the family. The surfaces are derived from collections $\mb{S}$ of spheres that touch tangentially, by taking the disjoint union of the components of the collection that are within a fixed small distance from the points of tangency with the components of the closed complement. The collection $\mc{S}$ of components that are far from the points of tangency are regular subsets of spheres, and nearby CMC surfaces are freely prescribed by perturbations of the boundary. Each element in the collection  $\mc{N}$  of components  near the points of tangency  lies at the extremal limit of  the family of constant mean curvature necks described in Section \ref{PreEmSpheres}. The  family of surfaces $\mc{C}$ is defined to be the disjoint union of the components  in $\mc{S}$ and $\mc{N}$ and is  controlled by  parameters  $\xi \in \ul{\xi} $ which parametrize  the family of constant mean curvature  necks, as well as boundary perturbations of the components of $\mc{S}$. We group the parameters into two types, according to the jacobi fields they generate in $\mc{C}$. The type II parameters can be  roughly thought of as dirichlet parameters with $C^{2, \alpha}$ boundary data, and generate $C^{2, \alpha}$, non-trivial jacobi fields on $\mc{C}$.  The full complement of jacobi fields is not generated by type II variations, as those that are generated by $\mc{N}$ through variations of the family of CMC necks are orthogonal to the space of lower modes. The type I parameters behave differently, either generating trivial jacobi fields or else singular ones with logarithmic poles at the points of tangency. These parameters modify the scale and rotate the axes of the necks in $\mc{N}$. The type I and type II parameters will play distinct roles in the construction.   We also define a pairing $*$--an order two diffeomorphism--of the boundary $\partial \mc{C}$ by the condition that the paired components project to the same curve under the canonical projection into $\mb{R}^3$. This gives rise to a notion of even and odd functions on the boundary of $\mc{C}$, which we use to formulate the Cauchy data matching operator.    In Section \ref{TheDefectOperator}, we define an operator $\Lambda$, which we call the \emph{defect operator}, associated to the family $\mc{C}$ which we use to match the Cauchy data of components of $\mc{C}$ \emph{up to translations}. That is, the zeroes of $\Lambda$ correspond to surfaces in the family whose paired boundary components differ by a translation and whose co-normals are opposing. Roughly speaking, the operator $\Lambda$ measures the angle between conormals across paired boundary components. It maps from the parameter space $\ul{\xi}$ into the error space $\ul{e}_*$ which is equal to the space of even $C^{1, \alpha}$ functions on $\partial \mc{C}$. We also prove differentiability and defined-ness of $\Lambda$ on a ball about the origin in the parameter space $\ul{\xi}$. Since the components of the  nonsingular part $\mc{S}$ of $\mc{C}$ are regular, we can restrict the family $\mc{C}$ apriori to variations in which paired boundary components differ by translations. In other words, we define the family of surfaces $\mc{C}$ so that at each parameter value  $\xi$,  the dirichlet data of $\mc{C}_{\xi}$ on paired boundary components is matched up to a translation. This reduces the Cauchy data matching problem to finding surfaces with opposing conormals and to resolving the translational errors separately.    In Section \ref{TypeIIVariations}, we study the linearization of $\Lambda$ in the type II parameters of the family. The operator is shown to be self-adjoint, and thus index zero, and the kernel--which we denote by $\ul{e}_{* ,I}$-- corresponds to the space of bounded jacobi fields on $\mb{S}$, with each sphere in $\mb{S}$ contributing a three dimensional space of translational jacobi fields in the absence of imposed symmetries. Thus, the orthogonal complement $\ul{e}_{*, II}$  of kernel the  $\ul{e}_{*, I}$ is generated by the type II parameters.  In Section \ref{SpheresLattice}, we specialize to the case that the collection $\mb{S}$ of spheres, from which the family $\mc{C}$ is derived, is arranged on a regular integer lattice in the plane. The collection is then invariant under a group $\mc{G}$ of symmetries, which is large enough so that the subspace of $\mc{G}$-invariant parameters of the family $\mc{C}$ avoids the co-kernel of the Type II variations. Thus, $\mc{G}$-invariant perturbations of $\mc{C}$ through zeroes of $\Lambda$ are freely prescribable through variations of $\mc{G}$-invariant Type $I$ parameters. Moreover,  the symmetries of the $\mc{G}$-invariant perturbations imply that the remaining translational error is trivially resolved. In Section \ref{DefectOnDelaunay}, as a prelude to addressing symmetry breaking variations, we record the $L^2$ projections of the variations of  $\Lambda$  generated by  type I and type II parameters at small $\mc{G}$-invariant perturbations of $\mc{C}$. The symmetries of $\mc{C}$, and an imposed orthogonality condition on the necks, together imply that the type II variations preserve the orthogonal complement of $\ul{e}_{*, I}$, and the type I and type II together generate the   error space $\ul{e}_{*}$.  In fact, to generate the co-kernel, only type I variations of the \emph{vertical necks} in $\mc{N}$--necks that connect spheres belonging to the same column of $\mb{S}$--are needed. We denote by $\ul{e}$ the space  all type II parameters, together with the type I parameters supported on the vertical necks. The entire parameter space $\ul{\xi}$ is then the direct sum of $\ul{e}$ with the space $\ul{v}$ of type I parameters supported on the horizontal necks. In Section \ref{SymBreaPert}, we develop the linear theory needed for symmetry breaking perturbations. Of note here is our use of \emph{graded norms}, in which terms are measured not by their absolute size, but by their differences across fundamental domains. We then show that  $\Lambda$ is defined and differentiable in a ball about zero in the parameter space equipped with an exponentially decaying graded norm, and that the linearization in $\ul{e}$ is a bounded isomorphism onto $\ul{e}_*$. By appealing directly to the implicit function theorem, we are then able to construct families $\mc{C}_{v}$ of zeroes of $\Lambda$ parametrized by the remaining free parameters $v \in \ul{v}$ of the construction. In Section \ref{CanTranErr}, we show that the translational errors  between components of the  surfaces $\mc{C}_{v}$ are resolvable with the remaining free parameters. This is done by arranging $\mc{C}$ into sub-collections we call \emph{branches}, in which translational errors are trivially resolvable. Basically,  a branch comprises the components of $\mc{S}$ belonging to a single column, along with all incident half necks. The surfaces in each branch can then be translated independently from each other so that paired boundary components are matched exactly, not just up to translations. Since the half necks are translated independently from each other, this process introduces new translational errors--namely, the differences between the waists of the half necks belonging to the same CMC neck--which need to be resolved. This error is measured by an operator which we denote by $\mc{E}$ which at each horizontal neck takes values on $\mb{R}^2$--due to the imposed symmetries. We then show that the remaining free parameters generate the error space of $\mc{E}$ at each neck. To match the surfaces, we restrict our attention to the finite dimensional subspace $\ul{v}_0$ of $\ul{v}$ comprising parameters $v \in \ul{v}$ that  vanish away from a fixed compact subset $\mb{F}$ of horizontal necks. The operator $\mc{E}$ is then a bounded isomorphism near zero. By translating the branches independently from each other, and correcting the error with $v$, we then obtain families of smoothly immersed CMC surfaces prescribed by the relative translations $d$ of the branches. Finally, in Section \ref{EmbeddSol}, we exhibit explicit choices of $\mb{F}$ and $d$ that give rise to embedded surfaces with prescribed genus. Since the construction just described only gives rise to surfaces with prescribed even numbers of ends greater than three, we explain  a modification that yields any number of ends greater than four.

\section{CMC necks} \label{CMCNecks}

The central tool used in our construction is a family of CMC annuli constructed in \cite{Kleene} whose relevant properties we now recall. 

\subsection{Catenoidal necks}
The  standard catenoid is given implicitly by the  equation $r = \cosh(z)$ in cylindrical coordinates on $\mb{R}^3$.  Given $\tau > 0$, let $\mc{N}_{\tau}$ denote the intersection of the standard catenoid the tube of radius $r =1$ about the $z$-axis. We refer to any surface of the form $\mc{N}_{\tau}$ as a catenoidal  neck, and we refer  to $\tau$ as the  scale. Each catenoidal neck  is invariant under rotations about the $z$-axis and reflections through the plane $\{z = 0 \}$.  The boundary $\partial \mc{N}_{\tau}$  of $\mc{N}_{\tau}$  comprises the circles $\partial^{\pm} \mc{N}_{\tau} : = \mb{S}^2 \pm l_{\tau}e_z$, where here $\mb{S}^2$ denotes the unit circle in the plane $\mb{R}^2 = \{ z = 0\}$ and where we have set $l_{\tau} : = \tau \text{Arcosh} \left(\frac{1}{\tau}\right) $.  The waist $w_{\tau}$ of the catenoidal neck $\mc{N}_{\tau}$ is the intersection of $\mc{N}_{\tau}$ with the plane $\{ z = 0\}$. It coincides with the circle of radius $\tau$ centered at the origin in $\mb{R}^2$.  At the boundary component $\partial^+ \mc{N}_{\tau}$ the outward conormal is given by $ \cos(\tau) e_r + \sin(\tau) e_z$. By symmetry, the outward conormal at $\partial^{-}\mc{N}_{\tau}$ is given by $ \cos(\tau) e_r - \sin(\tau) e_z$. Thus, the parameter $\tau$ has the dual interpretation as the \emph{flux}, or more precisely  the \emph{mean flux} of the neck $\mc{N}_{\tau}$ in the sense that  
 \[
 \int_{\partial^+ \mc{N}_{\tau}} \eta = 2\pi\sin(\tau) e_z,
 \]
 where above $\eta$ denotes the outward conormal to the surface $\mc{N}_{\tau}$. Thus, we will alternatively refer to $\tau$ as the flux  of the neck $\mc{N}_{\tau}$.
 
 \subsection{The separation parameter $\sigma$}
 It will be useful in places to parametrize the family of necks $\mc{N}_{\tau}$ by their separation, rather than their flux. We define the \emph{separation} of the neck $\mc{N}_{\tau}$ to be  half the distance between the top and bottom boundary components. Precisely, the separation $\sigma = \sigma_{\tau}$ as a function of $\tau$ is given by
 \begin{align}\label{SigmaFromTau}
 \sigma = \tau \text{Arcosh}\left( \frac{1}{\tau}\right).
 \end{align}
Clearly we  have that $\sigma \rightarrow 0$ as $\tau \rightarrow 0$ and $\frac{\sigma}{\tau} = \text{Arcosh}\left( \frac{1}{\tau}\right) \rightarrow \infty$.  Moreover,  we have $\sigma' = \text{Arcosh}\left( \frac{1}{\tau} \right) - \frac{1}{\sqrt{\tau^2 - 1}} \rightarrow \infty$ as $\tau \rightarrow 0$. 

\subsection{Smooth convergence at  $\tau = 0$}
The surfaces $\mc{N}_{\tau}$ are embedded minimal annuli depending smoothly on $\tau$ for $\tau > 0$.  
 Let $\mc{N}^{\pm}_{\tau}$ denote the intersection of $\mc{N}_{\tau}$ with the half space $\{ \pm z \geq 0\}$.  We refer to any surface of the form $\mc{N}^\pm_{\tau}$ as \emph{catenoidal half neck}. Clearly, the surfaces $\mc{N}_{\tau}^+$  and $\mc{N}^-_{\tau}$ converge in $C^k$ away from the origin to the disk $\mc{D}$ of radius $1$ in the plane as $\tau \rightarrow 0$. Moreover, the convergence is smooth modulo translations in the following sense: Let  $\tilde{\mc{N}}^\pm_{\tau}$ denote the translation of  $\mc{N}_{\tau}$  with boundary contained in the plane $\{ z = 0\}$, so $ \tilde{\mc{N}}^+_{\tau}= \mc{N}^+_{\tau} - l_{\tau} e_z$. Then the family of surfaces $\tilde{\mc{N}}^+_{\tau}$ extends smoothly to $\tau = 0$ away from the origin in $C^k$  with  $\tilde{\mc{N}}^+_{0} = \mc{D}$ and the variation field generated at $\tau = 0 $ on $\mc{D}$ is $ \log (r) e_z$. Again, by symmetry the family of surfaces $\tilde{\mc{N}}^-_{\tau}$ extends smoothly to $\tau = 0$ away from the origin in $C^k$ with $ \tilde{\mc{N}}^-_{0} = \mc{D}$ and the variation field generated at $\tau = 0 $ is $- \log (r) e_z$.
\subsection{CMC necks}
The family of surfaces constructed in \cite{Kleene} extend the properties of catenoidal necks discussed above. They are constructed as normal perturbations of catenoidal necks by functions solving a constant mean curvature dirichlet problem on catenoidal necks of small scale. We restate the main theorem of \cite{Kleene} below
\begin{theorem}\label{Prop:MainTheoremSimple}
There is $\epsilon> 0$ and $C > 0$ such that: Given $\tau > 0$ with $\tau < \epsilon$, a function $f \in C^{2, \alpha} (\partial \mc{N}_\tau)$ with $\| f\| < \epsilon$  and $\delta \in (-\epsilon, \epsilon )$, there is a unique normalized  function $\mc{h}_{\tau, \delta, f}$ in $C^{2, \alpha} (\mc{N}_\tau)$ satisfying the estimate
\[
|\mc{h}_{\tau, \delta, f}| \leq^C r^2\| f\|
\] 
and such that the normal graph $\mc{N}_{\tau, \delta, f}$ over $\mc{N}_\tau$ has constant mean curvature $\delta$ and such that the trace of $\mc{h}_{\tau, \delta, f}$ agrees with $f$ up to lower modes. 
 \end{theorem}
Here, the term \emph{lower mode} refers to  a function on a circle in the span of $1$, $\cos(\theta)$ and $\sin(\theta)$, where $\theta$ parametrizes the circle at constant speed with period $2\pi$.  A lower mode on $\partial \mc{N}_{\tau}$ is then a function whose restriction to each circular boundary component of $\mc{N}$ is a lower mode. Since the boundary of $\mc{N}_{\tau}$ is a pair of circles, the space of lower modes on $\partial \mc{N}_{\tau}$ is six dimensional, with three dimensions arising from each boundary component.  If $f$ is a function on a circle, we let $\ol{f}$ denote its projection onto the space of lower modes and $\mr{f}$ its projection onto the orthogonal complement, and we refer to $\ol{f}$ and $\mr{f}$ as the lower and higher parts of $f$, respectively. Clearly, these definitions extend to functions defined on collections of circles in the obvious way. The term \emph{normalized} appearing in the statement of Theorem \ref{Prop:MainTheoremSimple} means the following: Introduce coordinates $(\theta, s)$ on the neck $\mc{N}_{\tau}$, where $\theta \in \mb{R}$ parametrizes the foliating circles of $\mc{N}_{\tau}$ by constant speed with period $2\pi$ and $s$ parametrizes the profile curve of $\mc{N}_{\tau}$ at unit speed, and such that the slice $s = 0$ corresponds to the waist $w_\tau$ of $\mc{N}_{\tau}$. A function $u$ on $\mc{N}_{\tau}$ is normalized if both the restriction of $u$ to the waist $w_{\tau}$ as well as the conormal derivative $u_{s}$ across $w_{\tau}$, are orthogonal to the lower modes on $w_{\tau}$. Since the space of lower modes on $w_{\tau}$ is three dimensional, this represents a six dimensional family of constraints.  We will refer to any surface of the form $\mc{N}_{\tau, \delta, f}$ as a \emph{constant mean curvature neck}. Each CMC neck  $\mc{N}_{\tau, \delta, f}$ is an immersed  $C^{2, \alpha}$ surface with constant mean curvature $\delta$.

\subsection{Criteria for embeddedness of CMC necks}
  Since the functions $\mc{h}_{\tau, \delta, f}$ defining $\mc{N}_{\tau, \delta, f}$ are normalized, it follows  from the uniqueness assertion in Theorem \ref{Prop:MainTheoremSimple}  that the surface $\mc{N}_{\tau, \delta, f}$ is  invariant under reflections  through the plane $\{z = 0 \}$ when $f$ is, and is invariant under rotations about the $z$-axis when $f = 0$.  In this last case, the surfaces $\mc{N}_{\tau, \delta, 0}$ are rotationally symmetric about the $z$-axis and have constant mean curvature $\delta$. Thus, their images lie in Delaunay surfaces. Since we are orienting the catenoidal necks $\mc{N}_{\tau}$, and consequently the surfaces $\mc{N}_{\tau, \delta, f}$ by the `downward pointing' unit normal, for negative values of $\delta$ the surfaces lie in immersed family of \emph{nodoids} and are thus non-embedded for small values of $\tau$ relative to $\delta$,  and for positive  values of $\delta$ the surface $\mc{N}_{\tau, \delta, 0}$ lies the embedded family of \emph{unduloids} and is thus embedded for all positive values of $\tau$. Since the values of $f$ are unrestricted relative to the size of $\delta$ and $\tau$, the surfaces $\mc{N}_{\tau, \delta, f}$ are not in general embedded, even for negative values of $\delta$. However,  for positive $\delta$, $\mc{N}_{\tau, \delta, f}$ is embedded provided $f$ is sufficiently small relative to $\delta$. 
\subsection{Continuity of half necks at $\tau$ = 0}
Let $\mc{D}$ denote the unit disk in the plane $\{ z = 0\}$ centered at the origin. For small values of $\delta$ and $f \in \mr{C}^{2, \alpha} (\mb{S})$, there is  (Corollary 7.2 in \cite{Kleene}) a unique function $\mc{h}_{\delta, f}$ with
\[
H_{\mc{D}} (\mc{h}_{\delta, f}) = \delta, \quad \mr{\mc{tr}} (\mc{h}_{\delta, f}) = 0,
\]
and satisfying the normalization $\mc{h}_{\delta, f} (0, 0) = \nabla \left. \mc{h}_{\delta, f} \right|_{(0, 0)} = 0$. We let $\mc{D}_{\delta, f}$ denote the normal graph over $\mc{D}$ by the function $\mc{h}_{\delta, f}$, where here we are orienting $\mc{D}$ with the downward pointing unit normal $- e_z$. The $\mc{D}_{\delta, f}$ is an embedded disk with constant mean curvature $\delta$ depending smoothly on $\delta$ and $f$. 

 Let $\mc{h}^{\pm}_{\tau, \delta, f}$ denote the restriction of $\mc{h}_{\tau, \delta, f}$ to $\mc{N}^{\pm}_{\tau}$ and $\mc{N}^{\pm}_{\tau, \delta, f}$ the normal graph over $\mc{N}^{\pm}_{\tau}$ by $\mc{h}^{\pm}_{\tau, \delta, f}$. We will refer to any surface of the form $\mc{N}^{\pm}_{\tau, \delta, f}$ as a \emph{CMC half neck}. By the uniqueness assertion in Corollary 7.2 and the estimate for $\mc{h}_{\tau, \delta, f}$ in Proposition \ref{Prop:MainTheoremSimple}, the function $\mc{h}^+_{\tau, \delta, f}$ converges to $\mc{h}_{\delta, f_+}$ in $C^{2, \alpha/2}$ away from the origin, where  $f_+$ denotes the restriction of $f$ to $\partial^+ \mc{N}_{\tau}$. Similarly, the function   $\mc{h}^-_{\tau, \delta, f}$ converges to $\mc{h}_{-\delta, -f_-}$, where again  $f_-$ denotes the restriction of $f$ to $\partial^- \mc{N}_{\tau}$. The family of half necks $\mc{N}^\pm_{\tau, \delta, f}$ then extends continuously to $\tau= 0$ in $C^{2, \alpha/2}$ away from the origin.

\subsection{Differentiability at $\tau = 0$}
The family of half necks $\mc{N}^\pm_{\tau, \delta, f}$ extends continuosly  but not differentiably to $\tau = 0$ with $\left. \mc{N}^+_{\tau, \delta, f} \right|_{\tau = 0} = \mc{D}_{\delta, f_+}$ and $\left. \mc{N}^-_{\tau, \delta, f} \right|_{\tau = 0} = \mc{D}_{-\delta, -f_-}$.  However, similarly to the family of catenoidal necks,  the  extended family CMC half necks is differentiable in $\tau$ at $\tau = 0$ \emph{modulo translations}. That is, the family $\tilde{\mc{N}}_{\tau, \delta, f}^{\pm} : = \mc{N}^{\pm}_{\tau, \delta, f}  -\pm l_{\tau} e_z $ is differentiable in $\tau$ at $\tau = 0$, and smooth in $\delta$ and $f$. The  variation field generated by the family in $\tau$ at $\tau = f  = \delta = 0$ is again $ \log (r) e_z$. 

\subsection{Conventions in this article}
In this section we describe some minor modifications to the parameters that describe the family of CMC necks  that will be more convenient for our use in this article. The basic problem is that the surfaces $\mc{N}_{\tau, \delta, f}$ as introduced are described as graphs over the family of minimal catenoidal necks $\mc{N}_{\tau}$, which is slightly inconvenient in the context of constant mean curvature surfaces.  We thus take some time to modify  the definitions of these parameters so that certain quantities, and their geometric interpretations, are defined relative to the constant mean curvature neck $\mc{N}_{0, \delta, 0}$ instead. 

\subsubsection{Rescaling of the surfaces}
We will, throughout this article regard $\delta > 0$ as a  small but fixed constant and suppress it from our notation. Thus, for example, we will write $\mc{N}_{\tau, f}$ to mean $\mc{N}_{\tau, \delta,f}$ and $\mc{D}_f : = \mc{D}_{\delta, f}$.  We will also set $\mc{N}_{\tau} : = \mc{N}_{\tau, \delta}$ and $\mc{D} = \mc{D}_{\delta}$. This can potentially cause confusion, and to avoid this we observe that for example $\mc{N}_{\tau}$ now longer refers to the family of minimal surfaces defined above, but rather the CMC annulus of rotation $\mc{N}_{\tau, \delta}$. Similarly, $\mc{D}$ no longer refers to the unit disk in the plane, but rather the rotational CMC graph over $\mc{D}$ with constant curvature $\delta$ and tangent to the origin.  We will also assume that the surface $\mc{N}_{\tau, \delta, f}$ has been scaled  so that  the mean curvature is $1$. Precisely, we will identify $\mc{N}_{\tau, f}$ with the scaled surface $\delta \cdot \mc{N}_{\tau, \delta,  f}$. Observe that, when $f$  and $\tau$ vanish, the surface  $\mc{N}_0 : = \left. \mc{N}_{\tau, f} \right|_{\tau = f = 0}$ is the intersection of two spheres of radius $1/2$ centered at the points $(0, 0, \pm 1/2)$ with  a ball of small radius about the origin.  Thus, it is the union of the two constant mean curvature disks $\mc{D}_{\pm} : = \mc{D}_{\mp \delta}$.

\subsubsection{Modifying the Dirichlet parameter $f$}
 If we let $\mc{tr} (u)$ denote the trace of the function $u$, we then have that $\mr{\mc{tr}}(\mc{h}_{\tau, \delta, f}) = \mr{f}$, or equivalently $\mc{tr}(\mc{h}_{\tau, \delta, f}) = f + \ell$, where $\ell$ is a lower mode on $\partial \mc{N}_{\tau}$ depending smoothly on $\delta$ and $f$ for all values of $\tau$, smoothly in $\tau$ for $\tau \neq 0$ and differentiably in $\tau$ at $\tau = 0$. Moreover, it holds that $\frac{\partial \ell}{\partial f} = 0$ at $f = 0$.  The normal variation field on $\mc{N}_{\tau, f}$ generated by $\dot{f}$at $f= 0$ is  then
\[
\left(\frac{\partial \mc{N}_{\tau, f}}{\partial f} (\dot{f})\right)^\perp =  \delta \beta_{\tau} \dot{\mc{h}}_{\mc{N}_{\tau}, \dot{f}}.
\]
where above $\beta_{\tau} : = - N_{\mc{N}_\tau} \cdot e_z$ is the cosine of the angle between the unit normal $\mc{N}_\tau$ and $-e_z$.   Throughout this article, we will re-paramatrize the family of CMC necks slightly be setting $\mc{N}_{\tau, \delta, f} = : \mc{N}_{\tau, \delta, f/\beta_0}$. With this convention, the variation field $\dot{\mc{h}}_{\mc{N}_{0},\dot{f}}$ is a jacobi field on $\mc{N}_0$ with trace $\delta f$, rather than $\delta \beta f$.

\subsubsection{Modification of the flux parameter $\tau$ and the separation parameter $\sigma$}
Let $\lambda = \lambda_{\tau, f}$ denote the angle between the outward conormal of $\mc{N}_{\tau, f}$ and $\mc{N}_{0}$, which we regard as a function on $\partial \mc{N}_0$. Precisely, if $\eta = \eta_{\tau, f}$ and $\eta_0$ denote the outward conormals of  $\mc{N}_{\tau, f}$ and $\mc{N}_{0}$, respectively, then $\lambda$ denotes the angle between $\eta'$ and $\eta_0$, where here  $\eta'$ denotes the projection of $\eta$ into the plane spanned by $\eta_0$ and $N_0$, the unit normal to $\mc{N}_0$. For clarity, we remark that we are taking $\langle \eta_0, N_0 \rangle$ as a positive basis and thus 
\[
\frac{\eta'}{|\eta'|} =  \cos (\lambda)\eta_0 + \sin(\lambda)N_0
\]
By continuity, for small values of $f$, $\tau$ and $\delta$, the projection $|\eta'|$ is non-zero and thus $\lambda$ is a  well defined  differentiable function of $\tau$, $f$ and $\delta$ near zero. When $\delta$ and $f$ vanish,  the functions $\mc{h}_{\tau, 0, 0} $ vanish as well, and thus $\mc{N}_{\tau,  0, 0} = \mc{N}_{\tau}$. Since $\lambda$ is invariant under translations of the surface,    the variation $\dot{\lambda} : = \left. \frac{\partial \lambda}{ \partial \tau} \right|_{\tau = f = \delta = 0}$ of $\lambda$ in $\tau$ at $\tau  = f= \delta = 0$  is given along the boundary component $\partial^+\mc{N}_0$ by
\[
\dot{\lambda}   = N_0 \cdot V^\perp + \nabla_{V^\top} \eta_0.
\]
Here $V : = \left. \frac{d}{d\tau} \tilde{\mc{N}}_{\tau}^+ \right|_{\tau = 0}$ denotes the normal variation field generated by $\tau$ of the family $\tilde{\mc{N}}^{+}_{\tau}$. When $\delta  = 0$, we have that $\tilde{\mc{N}}^+ = \mc{D}$--here $\mc{D}$ again denotes the unit disk on $\mb{R}^2$--and thus the variation field $V = \log (r) e_z  $ is purely normal. Thus, at $\delta = 0$ it holds that $\dot{\lambda} = -1$. Since $\lambda$ is differentiable in $\delta$ for small $\delta$ we have that $\dot{\lambda} = - \beta$, where $\beta > 0$ is a constant depending continuously on $\delta$ and approximately equal to $1$ (note that this is not the same $\beta$ from above). We thus re-parametrize the family $\mc{N}_{\tau, f}$ by setting $\mc{N}_{\tau, f} : = \mc{N}_{\tau/\beta, f}$. With this re-parametrization we have that $\dot{\lambda} = -1$ for our fixed choice of $\delta$. With this in mind, we also choose to redefine the relative separation parameter so that $\sigma = \sigma_{\tau}$ is given instead by
\begin{align}\label{NewSigmaFromTau}
\sigma = l_{\tau} : = \beta\tau \text{Arcosh}\left(\frac{1}{\beta \tau} \right).
\end{align}

\section{Pre-embedded collections of spheres} \label{PreEmSpheres}
Let $\mb{S}$  denote a collection of spheres with radius $\frac{1}{2}$. The \emph{singular set} $\mc{s}$ of $\mb{S}$ is the collection of points belonging to more than one sphere.  We assume that the singular set is a discrete set of points, and we say that  a  collection $\mb{S}$ with this property is   \emph{pre-embedded}.

Fix   a constant $\delta > 0$ sufficiently small and set  $\mc{N} : = \mb{S} \cap \ol{B}_{\delta} (\mc{s})$, where $\ol{B}_{\delta} (\mc{s})$ denotes  closed ball of radius $\delta$ about $\mc{S}$. Then for $\delta$ sufficiently small, $\mc{N}$ is a collection of rotationally invariant constant mean curvature necks with vanishing flux, and we refer to $\mc{N}$ as the \emph{singular part} of $\mb{S}$.

\begin{remark}
Observe that here $\delta$ differs slightly from the $\delta$ used to construct the family of constant mean curvature necks in Theorem \ref{Prop:MainTheoremSimple}, and that we have fixed as a background constant throughout the article. Denoting it here by $\delta'$, the relationship between the two is given by $\delta'= \delta \sqrt{1 - \delta^2}$.
\end{remark}

 We refer to the  closed complement of $\mc{N}$, $\mc{S} : =  \mb{S} \setminus B_{\delta} $ as the \emph{spherical part} of $\mb{S}$. Each component of $\mc{S}$  up to translations is a sphere of radius $1/2$ with finitely many extrinsic disks of small radius $r_0$ removed. We will refer to any such surface as a \emph{spherical domain}. Since the collection $\mb{S}$ is pre-embedded, the number such disks, as well the pairwise distance from their centers, is bounded by a universal constant, and the components of $\mc{S}$ lie in a fixed compact subset of  the moduli space of compact CMC surfaces with smooth boundary, independent of the collection $\mb{S}$. As a consequence, taking $r_0 > 0$ smaller if necessary,  and since the stability operator on the sphere has kernel--the three dimensional space spanned by the translational killing fields--we can ensure by strict monotonicity of eigenvalues that each spherical domain  is non-degenerate, meaning that the Dirichlet kernel of the stability operator is trivial. We will assume throughout that $\delta > 0$ has been chosen sufficiently small so that this is the case

\subsection{Pairings}
Let $\mc{C}$ denote the disjoint union of the components of the collections $\mc{N}$ and $\mc{S}$, and let $\pi: \mc{C} \rightarrow \mb{R}^3$ denote the canonical  projection into $\mb{R}^3$. Then $\pi \mc{C} = \mb{S}$ and the projection restricts to a diffeomorphism on the interior of $\mc{C}$. On the boundary $\pi: \partial \mc{C} \rightarrow \mb{R}^3$ is two to one and thus induces an order two diffeomorphism of the boundary $*: \partial \mc{C} \rightarrow  \partial\mc{C}$--here order two means that the square $**$ of $*$ is the identity--determined by the condition.
\[
\pi * =  \pi.
\]
Alternatively, since the canonical projection restricts to a $1-1$ mapping on $\mc{S}$ and $\mc{N}$--or more precisely, their images in the disjoint union, we can interpret $*$ as a diffeomorphism of $\partial \mc{S}$ with $\partial \mc{N}$.
We use $*$ to define a notion of even and odd functions on $\partial \mc{C}$. A function $f$ on $\partial \mc{C}$ is said to be \emph{even} (resp. \emph{odd}) if it holds that $f^* = f$ (resp. $f^* = - f$). The even part and odd part of a function $f$ are respectively  given by
\[
f^e = \frac{f + f^*}{2}, \quad f^o = \frac{f^* - f}{2}.
\]
Clearly, the even and odd parts of a function are respectively even and odd, and it holds that $f = f^e + f^o$. 
\subsection{The augmented family $\mc{C}$}
The surface $\mc{C}$ lies in a larger family of surfaces  which we now describe. Define a \emph{piecewise motion} of a surface $S$ in $\mb{R}^3$ to be  mapping $\rho: S \rightarrow \mb{R}^3$ such that the restriction of $\rho$ to each component of $S$ is a rigid motion. We can add the terminology \emph{piecewise translation} and \emph{piecewise rotation} in the obvious way, and generalize these notions to disjoint unions of surfaces in $\mb{R}^3$.

\subsubsection{Phase parameters}
We will need the following basic observation recorded  in the following Lemma, whose proof we omit.  In the following, we let $\phi$ be a piecewise rotation of $\mc{N}$ and $\mc{N}_{\phi}$ its image.   We will refer to any such choice of $\phi$ as a \emph{phase assignment} in $\mc{N}$ or in $\mc{C}$, and we let $\ul{\phi}$ denote the space of all phase assignments. \begin{lemma} \label{ParallelPhase}
For each $\phi \in \ul{\phi}$ near $0$ there is a unique collection $\mc{S}_{\phi}$ of spherical domains, with $\mc{S}_{\phi}$ depending smoothly on $\phi$ and such that: for each component $\partial$ of $\partial \mc{N}_{\phi}$, the component $\partial^* \subset \partial \mc{S}_{\phi}$ differs from $\phi$ by a translation. 
\end{lemma}

For a fixed phase assignment $\phi \in \ul{\phi}$ small relative to the constraints of Lemma \ref{ParallelPhase}, we let $\mc{C}_{\phi}$ denote the disjoint union of $\mc{S}_{\phi}$ and $\mc{N}_{\phi}$. 
\subsubsection{Spherical parameters}
Since  each component the spherical part $\mc{S}$ lies in a fixed compact subset of the space of  non-degenerate, compact CMC surfaces,  CMC surfaces near $\mc{S}$ are uniquely determined by their boundaries, which can be freely prescribed. That is:

\begin{lemma} There is $\epsilon > 0$   such that: Given a $C^{2, \alpha}$ vector field $X$ on $\partial \mc{S}$ with $\| X\| < \epsilon$, there is a unique CMC surface $\mc{S}_X$ near $\mc{S}$ such that 
\[
\partial \mc{S}_X = \partial \mc{S} + X
\]
Moreover, the mapping $X \mapsto \mc{S}_X$ is smooth in $C^{2, \alpha}$.
\end{lemma}
We let $\ul{\rho}$ denote the space of piecewise translations of $\partial \mc{S}$. 
For reasons which will become later in the article, we prefer to  parametrize these surfaces as follows:

\begin{lemma}
There is $\epsilon > 0$ such that: Given a  piecewise translation $\rho \in \ul{\rho}$ of $\partial \mc{S}$ with $\| \rho\| < \epsilon$, there is a unique CMC surface $\mc{S}_{\rho}$ with boundary given by
\[
\partial \mc{S}_{\rho} =\left( \mb{I} + \rho \right) \cdot \partial \mc{S}.
\]
where above we have used $\mb{I}$ to denote the identity translation.  Moreover, by taking $\rho$ smaller, if necessary, we can assume that the surface $\mc{S}_{\rho}$ is again non-degenerate
\end{lemma}
We can then define the surface
\[
\mc{S}_{\rho, X} = \mc{S}_{X + \rho}
\]
for $\rho$ sufficiently close to the identity and for $X$ sufficiently small in $C^{2, \alpha}$. The mapping $(\rho, X) \mapsto \mc{S}_{\rho, X}$ is smooth. 

\begin{lemma} \label{lael}
The normal variation field generated on $\mc{S}$ generated by $(\dot{\rho}, \dot{X})$ at $\rho = X = 0$ is $\dot{\mc{h}}_{S, \dot{X}^\perp + \dot{\rho}^\perp}$
\end{lemma}
\begin{proof}
Fix $\dot{\rho}$ and $\dot{X}$ and consider the one parameter family $\mc{S}_t : = \mc{S}_{t \dot{\rho}, t \dot{X}}$. Then $\mc{S}_t$ is a composition of smooth mappings and is thus a smooth one parameter family of surfaces. Thus, there is a smooth one parameter family of vector fields $V_t$ on $S_0$ such that $S_0 + V_t = S_{t}$. Observe that $S_t$ restricts as a smooth variation $\partial S_t$ of the boundary of $S_t$ and thus $\partial S_0 + \left. V_t\right|_{\partial S_0} = \partial S_t  = \partial + X + \rho$. Thus, the normal variation field of the family $\partial S_t$ along $\partial S_0$  is given by $\dot{X}' + \dot{\rho}'$, where here we have used $V'$ to denote the orthogonal projection of a vector field $V$ along $\partial S_0$ away from $\partial S_0$. Thus, the normal variation field $u$ of the family $S_t$ along $S_0$ has trace $\left(\dot{X}' + \dot{\rho}'\right)^\perp = \left(\dot{X} + \dot{\rho}\right)^\perp $. Since $u$ is a jacobi field along $S_0$ we have $u  = \dot{\mc{h}}_{S, \dot{X}^\perp + \dot{\rho}^\perp}$. This completes the proof. 
\end{proof}

\subsubsection{Dirichlet parameters}
The boundary of $\mc{C}$ is a disjoint union of circles, and thus the notions of  lower and higher mode extend easily to $\partial \mc{C}$. We declare a function on $\partial \mc{C}$ to be a \emph{lower mode} if its restriction to each component is a lower mode. Similarly  a function is \emph{higher mode} if its restriction to each boundary component is orthogonal to the lower modes. We let $\mr{C}^{k, \alpha}\left( \partial \mc{C}\right)$ denote the space of class $C^{k, \alpha}$ higher modes on $\partial \mc{C}$. We can then define the surface $\mc{N}_{\tau, f}$  component-wise as follows. Namely, given a function $f \in \mr{C}^{k, \alpha}\left( \partial \mc{N}\right)$ and constants $\tau$ indexed by the components of $\mc{N}$, we let $\mc{N}_{\tau, f}$ denote the collection of catenoidal necks given by
\[
\left(\mc{N}_{\tau, f}\right)_i = \mc{N}_{\tau_i, f_i}
\]
where $i$ indexes the components of   $\mc{N}$ and $f_i$ denotes the restriction of $f$ to $\partial \mc{N}_i$. Given a higher mode $f \in \mr{C}^{2, \alpha} (\partial \mc{N})$, define the vector field $X = X_{\tau, f} \in \partial \mc{N}$ by
\begin{align}\label{VecXDef}
X_{\tau, f} : = \left( \mc{tr}(\mc{h}_{\tau, f}) - \mc{tr}(\mc{h}_{\tau, 0}) \right)\left.  N_{\mc{N}_{\tau}} \right|_{\partial},
\end{align}
where above $\left.  N_{\mc{N}_{\tau}} \right|_{\partial}$ denotes the restriction of the unit normal $N_{\mc{N}_{\tau, 0}}$ to $\partial \mc{N}_{\tau}$. We identify $X_{\tau, f}$ with its even extension to $\partial \mc{C}_0$

 \subsubsection{The augmented family}
We can now define the family $\mc{C}: \xi \mapsto \mc{C}_{\xi}$. We first define the parameter space $\ul{\xi}$. We let $\ul{\phi}$ denote the space of piecewise rotations of $\mc{N}$, $\ul{\rho}$ the space of piecewise translations of $\partial \mc{S}$, $\ul{\tau}$ the space of flux parameters of $\mc{N}$, and $\ul{f}$ the space of even functions in $\mr{C}^{2, \alpha}(\partial \mc{C})$. Observe that, since even functions are determined by their restrictions to $\partial\mc{N}$ and $\partial \mc{S}$,  we have the natural identification  of $\ul{f}$ with $\mr{C}^{2, \alpha} (\partial \mc{S})$ and $\mr{C}^{2, \alpha} (\partial \mc{N})$.  

\begin{definition} \label{ParameterSpaces}
We set $\ul{\xi} : = \ul{\tau} \times \ul{ \phi} \times \ul{ \rho} \times \ul{f}$, and refer to $\ul{\xi}$ as the parameter space of the family $\mc{C}$. We also put $\ul{\xi}_{I}  = \ul{\tau} \times \ul{ \phi}$  and $\ul{\xi}_{II} : =  \ul{ \rho} \times \ul{ f}$, so that $\ul{\xi} = \ul{\xi}_{I} \times \ul{\xi}_{II}$. We will refer to parameters $(\tau, \phi) \in \ul{\xi}_I$ as type $I$ parameters and to parameters $(\rho, f)$ as type II.
\end{definition}

\begin{definition}\label{InitialSurfaces}
Given $\xi = (\tau, \phi, \rho, f) \in \ul{\xi}$,  and provided the following quantities are defined:
\begin{enumerate}
\item \label{InitialSurfaces1}We set $ \mc{N}_{\xi} = \mc{N}_{\tau, \phi, f}  : = \phi \cdot \mc{N}_{\tau, f}$, where  $\phi \cdot \mc{N}_{\tau, f}$ denotes the application of the piecewise rotation $\phi$ to $\mc{N}_{\tau, f}$.
\item  \label{InitialSurfaces2} We  set $\mc{S}_{\xi} = \mc{S}_{\tau, \phi, \mc{M }, f} : = \left(\mc{S}_{\phi}\right)_{\rho, X_{\tau, f}}$.

\item \label{InitialSurfaces3} We set $\mc{C}_{\xi} : = \mc{N}_{\xi} \coprod \mc{S}_{\xi}$.
\end{enumerate}
\end{definition}

\begin{proposition}
There is $\epsilon > 0$ such that: Given $\xi \in \ul{\xi}$  with $\| \xi\| < \epsilon$,  the surface $\mc{C}_{\xi}$ is defined and depends diffferentiably modulo translations on $\xi$. Moreover, it holds that $\left. \mc{C}_{\xi} \right|_{\xi = 0} = \mc{C}$. Finally, the diffeomorphism $*$ acts by translation on $\partial \mc{C}$; that is for any component $\partial \subset \partial\mc{C}$, it holds that $\partial$ and $\partial^*$ differ by a translation. 
\end{proposition}
\begin{proof}
The fact that $\mc{C}_{\xi}$ depends differentiably modulo vertical translations is a direct consequence of the fact that the each component is either a spherical domain, or else CMC neck.  Clearly, when $\xi = 0$ the surface $\mc{C}_{\xi}$ agrees with the initial surface $\mc{C}$. 
To see that paired boundary components $\partial$ and $\partial^*$ differ by a translation, observe first that the claim holds when $\tau$, $f$ and $\rho$ vanish, since then $\mc{C} = \mc{C}_{\phi} = \mc{N}_{\phi}\coprod \mc{S}_{\phi}$.  Since varying $\tau$ translates the boundary components of the catenoidal necks along a common axis, the claim holds for $\tau \neq 0$. Again, since $\partial \mc{S}$ and $\partial \mc{S}_{\rho}$ differ by the piecewise translation $\rho$, the claim holds for $\rho \neq 0$. Finally, since $X_{\tau, f}$  is even the claim holds for non-vanishing $f$ as well. This completes the proof.
\end{proof}

\begin{proposition}\label{FamilyLinearization}
The normal variation field generated by $(\dot{\rho}, \dot{f})$ on $\mc{S}$ at $\xi = 0$ is $\dot{\mc{h}}_{\mc{C}, \rho^\perp + \dot{f}}$.
\end{proposition}
\begin{proof}
By Lemma \ref{lael}  and Definition \ref{InitialSurfaces2} we have that that the normal variation field is given by $\dot{\mc{h}}_{\mc{S},\dot{\rho}^\perp + \dot{X}_{0, f}^\perp}$. From Definition \ref{VecXDef}, we have $\left(\dot{X}_{0, f}\right)^{\perp} = \delta \dot{f}$. This completes the proof. 
\end{proof}

\section{The defect operator}\label{TheDefectOperator}

\begin{definition}\label{NeumannOperator}
We let $\lambda = \lambda_{\xi}$ denote the angle between the outward conormals of $\mc{C}_{\xi}$ and $\mc{C}_{0}$. Precisely, if $\eta = \eta_{\xi}$ denotes the outward co-normal of  $\mc{C}_{\xi}$  then $\lambda_{\xi}$ denotes the angle between $\eta'_{\xi}$ and $\eta_0$, where here  $\eta'_{\xi}$ denotes the projection of $\eta_{\xi}$ into the plane spanned by $\eta_0$ and $N_0$, the unit normal to $\mc{C}_0$.Thus we have: 
\begin{align}\label{AngleFormula}
\frac{\eta'_{\xi}}{|\eta'_{\xi}|} =  \cos (\lambda_{\xi})\eta_0 + \sin(\lambda_{\xi})N_0.
\end{align}

\end{definition}
\begin{lemma} \label{NuemannOperatorProperties}
For each parameter value $\xi$, $\lambda_{\xi}$ lies in $C^{1, \alpha}(\partial \mc{C})$, and the mapping $\lambda: \xi \mapsto \lambda_{\xi}$ is once differentiable from $\ul{\xi}$ into $C^{1, \alpha}(\partial \mc{C})$. Moreover, the linearization $\dot{\lambda}$ of $\lambda$ at $\xi = 0$ is given by:
\[
\dot{\lambda}_{\dot{\xi}} =\left(\nabla V^\perp\right) \cdot \eta + A\left(V^\top, \eta \right)
\]
Above  $V$ denotes the variation field generated by $\dot{\xi}$, $V^\perp : = V \cdot N$ its normal part and $V^\top : = V - V^\perp N$ its tangential part. Also, $A$ denotes the second fundamental form of $\mc{C}_0$ and $N$ the unit normal. 
\end{lemma}

\begin{proof}
 Recall that the mapping $\mc{h}: (\tau, f) \mapsto \mc{h}_{\tau, f}$ is $C^1$  near $\tau = f = 0$, and thus the mapping$ (\tau, f) \mapsto \tilde{\mc{N}}^\pm_{\tau, f}$ is $C^1$ near $\tau = f = 0$. Since $\lambda_{\xi}$ is invariant under translations of the components of $\mc{C}$, it follows that $\xi \mapsto \lambda_{\xi}$ is  $C^1$ near $\xi = 0$.  That the codomain of $\lambda$ is the space $C^{1, \alpha} (\partial \mc{C})$, follows from standard regularity theory for elliptic operators. Namely, since the boundary  $\mc{C}$ is  class $C^{2, \alpha}$, then so is $\mc{C}$, and thus the outward conormal field is $C^{1, \alpha}$.
 
 Now, observe that if $S$ is a surface and $S_t$ is a variation of $S$ , then the  normal variation of the outward normal is given by
\begin{align}\label{UnitNormalVariation}
\dot{\eta} \cdot N = - \dot{N} \cdot \eta =\left(\nabla V^\perp\right) \cdot \eta + A\left(V^\top, \eta \right)
\end{align}
Since at $\xi = 0$ we have $|\eta_{0}'| = 1$ and otherwise $|\eta_{\xi}'| \leq 1$, the function $|\eta_{\xi}'|$ has a  maximum at $\xi = 0$  and thus the variation  of $\frac{\eta'_{\xi}}{|\eta'_{\xi}|}$ induced by $\dot{\xi}$ at $\xi = 0$ is given by
\[
\left(\frac{\eta'_{\xi}}{|\eta'_{\xi}|} \right)^{\cdot} = \dot{\eta},
\]
where above $\dot{\eta}$ denotes the variation of $\eta_{\xi}$ induced by $\dot{\xi}$ at $\xi = 0$. Differentiating and taking the dot product of both sides of (\ref{AngleFormula}) with $N_0$ and using (\ref{UnitNormalVariation}) gives
\[
\dot{\lambda} = \dot{\eta} \cdot N_0 = \left(\nabla V^\perp\right) \cdot \eta + A\left(V^\top, \eta \right)
\]
which is the claim. 
 \end{proof}
\subsection{The defect operator}

\begin{definition}
We let $\Lambda: \xi \mapsto \Lambda_\xi$ denote the even part of $\lambda$, so $\Lambda : = \lambda^{(e)}$. We refer to $\Lambda$ as the defect operator of $\mc{C}$.
\end{definition}

 In the following, we let $\ul{e}_*$ denote the space of even functions in  $C^{1, \alpha}(\partial \mc{C})$.
\begin{lemma} \label{ConormalMatching}
The mapping $\xi \mapsto \Lambda_{\xi}$ is a differentiable mapping of the ball $\{|\xi|  < \epsilon \}$ into $\ul{e}_{*}$. The linearization  $\dot{\Lambda}$ of $\Lambda$ at $\xi = 0$ is given by
\begin{align}\label{DotLambdaExpression}
\dot{\Lambda}_{\dot{\xi}} = \left( u_{\eta}\right)^{e}
\end{align}
where $u$ is the normal variation field generated by $\dot{\xi}$ on $\mc{C}$. Moreover, suppose that $\Lambda_{\xi} = 0$ for some $\xi$. Then  $\eta_{\xi} = \left(- \eta_{\xi} \right)^*$, where $\eta_{\xi}$ denotes the outward conormal to $\mc{C}_{\xi}$.
\end{lemma}

\begin{proof}
The differentiability of $\Lambda$ follows directly from the differentiability of $\lambda$. By Lemma \ref{NuemannOperatorProperties} we have
\begin{align*}
\dot{\Lambda}_{\dot{\xi}} & = \left( \left(\nabla V^\perp\right) \cdot \eta + A\left(V^\top, \eta \right) \right)^e  \\
& =  \left( \left(\nabla V^\perp\right) \cdot \eta \right)^{e} + \left(A\left(V^\top, \eta \right) \right)^e \\
& = \left( \left(\nabla V^\perp\right) \cdot \eta \right) + \left( \left(\nabla V^\perp\right) \cdot \eta \right)^* ,
\end{align*}
where above we have used that  the restriction of $V$ to the boundary is even, and the outward conormal $\eta$ to $\mc{C}$ is odd and thus
\[
  \left(A\left(V^\top, \eta \right) \right)^* =   \left(A\left(V^\top, \eta^* \right) \right) = -   A\left(V^\top, \eta \right).
\]
Setting $u = V^\top$ gives (\ref{DotLambdaExpression}).
If we let $\lambda_{\xi}$ denote the angle between $\eta_\xi$ and $\eta_{0}$ then $\Lambda_{\xi} = 0$ if and only if $\lambda_{\xi} = - \lambda_{\xi}^*$. Then we have 
\[
\frac{\eta'_{\xi}}{|\eta'_{\xi}|} =  \cos (\lambda_{\xi})\eta_0 + \sin(\lambda_{\xi})N_0.
\] 
and thus
\begin{align*}
\frac{\eta^{*'}_{\xi}}{|\eta'_{\xi}|}  & =  \cos (\lambda^*_{\xi})\eta^*_0 + \sin(\lambda^*_{\xi})N_0 \\
& = - \cos (\lambda^*_{\xi})\eta_0 - \sin(\lambda_{\xi}) N_0 \\
& = -\frac{\eta'_{\xi}}{|\eta'_{\xi}|} .
\end{align*}
Moreover, since $\partial$ and $\partial^*$ differ by translations, it then follows that $\eta_{\xi}  = - \eta^*_{\xi}$ as claimed. 
\end{proof}

\subsection{The linearized operator $\dot{\Lambda}$}
Let $\dot{\Lambda}: \ul{\xi} \rightarrow \ul{e}_*$ denote the linearization of $\Lambda$ at $\xi = 0$.  We will study the linearization in the type I and II parameters separately. The type II  parameters generate the error space $\ul{e}_*$ up to a finite dimensional co-kernel. Depending upon the geometry of the configuration $\mb{S}$, we will show that  the type I parameters generate the remaining co-kernel.

\section{Type II variations}\label{TypeIIVariations}
In this section we study the linearization  $\dot{\Lambda}$ of $\Lambda$ in the type II parameters by relating it to a  Dirichlet to Neumann operator.
\subsection{The Dirichlet to Neumann operator}

\begin{definition}\label{DtoNOpoerator}  We define the operator $\slashed{\partial}: C^{2,  \alpha} (\partial \mc{C}) \rightarrow  C^{1,   \alpha} (\partial \mc{C}) $ to be the operator given by:
\[
\slashed{\partial} g = \frac{\dot{\mc{h}}_{\mc{C},g}}{\partial \eta} : =  \nabla \dot{\mc{h}}_{\mc{C}',g} \cdot \eta.  
\]
We refer to the operator $\slashed{\partial}$ as the \emph{Dirichlet to Neumann operator} of the collection $\mc{C}$.
\end{definition}
It is well-known that  $\slashed{\partial}$ is an elliptic  first order self adjoint pseudo-differential operator with principle symbol $-|\xi|$.  Moreover it is straightforward to verify that $\slashed{\partial}$ is self adjoint  with respect to $L^2$ inner product on $\partial \mc{C}$: Let $u$ and $v$ be jacobi fields on $\mc{C}$, then
\[
\int_{\mc{C}'} u Lv - v L u = \int_{\partial \mc{C}'} u v_{\eta} - v u _{\eta}= 0.
\]
When $\mc{tr}(u) = f$ and $\mc{tr} (v)= g$, the second equality above gives $\int_{\partial \mc{C}'} f \slashed{\partial }g - g \slashed{\partial} f = 0$. It then follows directly that: 
\begin{proposition}
The even part $\slashed{\partial}^{(e)}$ of $\slashed{\partial}$ is self adjoint on the subspace of even functions on $\partial \mc{C}$
\end{proposition}
\begin{proof}
Since $\slashed{\partial}$ is self adjoint, we have
\[
\langle f, \slashed{\partial }g \rangle = \langle \slashed{\partial }f,g \rangle. 
\]
Applying $*$ to both sides above gives
\[
\langle f^*, \left(\slashed{\partial }g\right)^* \rangle = \langle \left(\slashed{\partial }f\right)^*,g^* \rangle. 
\]
Since we are assuming $f$ and $g$ are even, we have $f= f^*$ and $g = g^*$ and thus summing the two equations above gives
\[
\langle f, \slashed{\partial }g + \left(\slashed{\partial }g\right)^* \rangle = \langle \slashed{\partial }f + \left(\slashed{\partial }f\right)^*,g \rangle. 
\]
which gives the claim. 
\end{proof}
 
Thus, $\slashed{\partial}$ is index zero and the obstruction to inverting the operator arises from its kernel. The following lemma relates $\dot{\Lambda}$ to the operator $\slashed{\partial}$.
\begin{lemma}
It holds that $\dot{\Lambda} (f, \rho) = \slashed{\partial}^{(e)}(f + \rho^\perp)$, where here we are identifying the piecewise translation $\rho$ of $\partial \mc{S}$ with its even extension to $\partial \mc{C}$.
\end{lemma}
\begin{proof}
Fix a component $\mc{S}_\alpha$ in $\mc{S}$ and let $\mc{N}_{\alpha} = \mc{N}_{\alpha, 1} \coprod \ldots \mc{N}_{\alpha, k}$ denote the components of $\mc{N}$ that are incident to $\mc{S}_{\alpha}$. That is, for each component $\partial$ of the boundary $\partial \mc{S}_{\alpha}$, there is a unique $i \in \{1, \ldots k \}$ such that $\partial^*$ is in $\mc{N}_{\alpha, i}$. After applying a piecewise translation to $\mc{N}_{\alpha}$, we can assume that the boundaries are matched, so that $\partial = \partial^*$ for all components $\partial$ of $\mc{S}_{\alpha}$.  Since the boundary components of $\mc{S}_{\alpha}$ are matched the normal variation field generated by $(\dot{\rho}, \dot{f})$ at $\xi = 0$  has even trace and on $\mc{S}_{\alpha}$ is given by $\dot{\mc{h}}_{\mc{S}_{\alpha}, \dot{\rho}^\perp + f}$. By expression (\ref{DotLambdaExpression}) in  Lemma \ref{ConormalMatching}, we then have $\dot{\Lambda}_{\dot{\rho}, \dot{f}} = \left(\left( \dot{\mc{h}}_{C, f  + \rho^\perp}\right)_{\eta}\right)^{(e)} = \slashed{\partial}^{(e)} (f + \rho^\perp)$.

\end{proof}

\subsection{The kernel of $\slashed{\partial}^{(e)}$}
The kernel of $\slashed{\partial}^{(e)}$ in the space of even functions can be characterized  as the space of jacobi fields on $\mb{S}$. To do this, it will be helpful to have a few definitions first:

\begin{definition}\label{DefMatched}
A function $u: \mc{C} \rightarrow \mb{R}$ is said to be Dirichlet matched if its trace $\mc{tr}(u)$ is even. It is Neumann matched if the deriviative  $\frac{\partial u}{\partial \eta}$ of $u$ with respect to the boundary conormal is odd. It is Cauchy matched if it is both Dirichlet and Neumann matched. 
\end{definition}
An obvious consequence of Definition \ref{DefMatched} is that Cauchy matched functions descend as $C^1$ functions under the canonical projection into $\mb{R}^3$. Moreover, by standard elliptic regularity, if $u$ is a Cauchy matched jacobi field, then it descends  to a smooth function on  $\mb{S}$ under the canonical  projection.  We record this observation in Lemma \ref{Prop:MatchedFunctions} below, whose proof we omit. 
\begin{lemma} \label{Prop:MatchedFunctions}
Let $u: \mc{C} \rightarrow \mb{R}$ be a Dirichlet matched function. Then the projection $\wh{u}: \mb{S} \rightarrow \mb{R}$, given by $\wh{u}( \pi'(p))$ = u (p) is a well defined $C^{0}$ function on $\mb{S}$. When $u$ is Cauchy matched $\wh{u}$, is $C^1$. 
\end{lemma}
We can define lifts  of functions on $\mb{S}$ as an inverse operator to projections of  Dirichlet matched functions  on $\mc{C}$.
\begin{definition}\label{DefLifts}
Given a function $u$ on $\mb{S}$, its \emph{lift} $\wc{u}: \mc{C}' \rightarrow \mb{R}$ to $\mc{C}'$ is given by $\wc{u}(p) = u (\pi'(p))$
\end{definition}

Clearly, we have:
\begin{lemma}
The lift $\wc{u}$ of a smooth function $u$ on $\mb{S}$ is a smooth function on $\mc{C}'$. Moreover, $\wc{u}$ is Cauchy matched and it holds that $\wh{\wc{u}} = u = \wc{\wh{u}} $. 
\end{lemma}
\subsection{The kernel of $\slashed{\partial}^e$}We are ready to characterize the kernel of $\slashed{\partial}^e$.

\begin{definition}
We let $\ul{\mc{k}}$ denote the space of functions on $\mb{S}$ such that: Given $\mc{k} \in \ul{\mc{k}}$, the restriction of $\mc{k}$ to any sphere in the collection $\mb{S}$ is translational jacobi field 
\end{definition}

\begin{lemma} \label{Prop:SKernel}
The kernel of the stability operator $L$ of $\mb{S}$ coincides with $\ul{\mc{k}}$.
\end{lemma}
\begin{proof}
It is a standard fact that the kernel of the stability operator on a sphere is three dimensional and is spanned by the translational jacobi fields.  Thus, if $u: \mb{S} \rightarrow \mb{R}$ is smooth jacobi field on $\mb{S}$, then its restriction to any sphere in the collection is  a jacobi field, and thus $u$ belongs to $\ul{\mc{k}}$. Clearly, any element of $\ul{\mc{k}}$ is jacobi field on $\mb{S}$.
\end{proof}

We let $\ul{g}$ denote the space of even functions in $C^{2,\alpha} (\mc{C})$, $\ul{g}_{II}$ denote the orthogonal complement of $\ul{e}_{*, I}$ in $\ul{g}$ and $\ul{e}_{*, II}$ the orthogonal complement of $\ul{e}_{*, I}$ in $\ul{e}_{*}$. 

\begin{proposition}\label{LPrimeKernel}
 The kernel of the operator $\slashed{\partial}^{(e)}$ coincides with the space $e_{*,I } : = \mc{tr}\left(\wc{\ul{\mc{k}}}\right)$.
\end{proposition}
              
\begin{proof}
Pick $g \in \ul{g}$ and suppose first that $\slashed{\partial}^e (g)  =0$. Then  $\dot{\mc{h}}_{\mc{C}, g}$ is Cauchy matched and thus by Lemma  \ref{Prop:MatchedFunctions} the projection  $\wh{\dot{\mc{h}}}_{\mc{C}, g}$ of $\dot{\mc{h}}_{\mc{C}', g}$ is a well-defined smooth jacobi field on $\mb{S}$. Thus, by Lemma \ref{Prop:SKernel},  it follows that $\wh{\dot{\mc{h}}}_{\mc{C}, g}$ is in $\ul{\mc{k}}$. This is equivalent to the statement that  $\dot{\mc{h}}_{\mc{C}, g}$ belongs to $\wc{\ul{\mc{k}}}$ or, equivalently, that $\mc{tr}(\dot{\mc{h}}_{\mc{C}, g})$ belongs to $ \ul{e}_{*,I}: = \mc{tr}\left(\wc{\ul{\mc{k}}}\right)$.
\end{proof}
As a direct consequence we have:

\begin{corollary}\label{LambdaPrimeIso}
The map $\slashed{\partial}^{(e)}$ is a bounded  isomorphism of $\ul{g}_{II}$ onto $\ul{e}_{*, II}$.
\end{corollary}

\subsection{Resolving solutions}
We  now show that the type $II$ parameters generate the space $\ul{e}_{*, II}$. We will need the following Lemma:
\begin{lemma}
 The mapping $\rho \rightarrow \rho^\perp$, taking a piecewise translation of $\partial \mc{S}$ to its normal component $\partial^\perp$ along $\mc{S}$  is an isomorphism  of $\ul{\rho}$ onto the space lower modes on $\partial\mc{S}$.
\end{lemma}
\begin{proof}
  This is easily established by considering each boundary component separately: On each component $\partial$ of $\partial \mc{C}$, we have that $\rho$ is in the span of $e_x$, $e_y$   and $e_z$.  We can assume, after  a rigid motion of  the coordinate axes that $\partial$ is a circle in the plane $z = 0$ and centered at the origin. Parametrizing $\partial$ by the angular parameter $\theta$,  the unit normal at $\partial$ is of the form
\[
N = \left(\delta \sqrt{1 -\delta^2}\right)e_{r}(\theta)  + \left(\frac{1}{2} - \delta^2 \right) e_z.
\]
where $e_r(\theta) = \cos(\theta) e_x + \sin(\theta) e_z$. The $x$, $y$ and $z$ components of $N_{\mc{S}}$ are then in the span of $\cos(\theta)$, $\sin(\theta)$, and $1$, which are precisely the lower modes on $\partial$. 
\end{proof}

\begin{remark}
Throughout the remainder of the article, we will identify the space $\ul{\rho}$ of piecewise translations of $\partial \mc{S}$ with the subspace of piecewise translations $\rho$ such that $\rho^\perp$ is orthogonal to the space $\ul{e}_{*, I}$.
\end{remark}

We then have:
\begin{corollary}\label{TypeIIISo}
The mapping $\dot{\Lambda}: \ul{\xi}_{II} \rightarrow \ul{e}_{*, II}$ is an isomorphism.
\end{corollary}

\begin{proof}[Proof of Corollary \ref{TypeIIISo}]
Since $\slashed{\partial}^{(e)}$ is self adjoint, its  image is orthogonal to the kernel space $\ul{e}_{*, I}$  and thus lies in $e_{*, II}$. Now, pick $ e_* \in \ul{e}_{*, II}$ and consider  the  solution $g \in \ul{g}_{II}$ to the equation
\[
\slashed{\partial}^{(e)} (g) = e.
\]
Write $g = \ol{g} + \mr{g}$, where $\ol{g}$ and $\mr{g}$ are the lower and higher parts of $g$, respectively. We then set $f = \mr{g} \in \ul{f}$ and we let $\rho$ be determined by the condition that $\rho^\perp = \ol{g}$. We then have
\[
\dot{\Lambda}_{f, \rho} = \left(\slashed{\partial} (f + \rho^\perp)\right)^{(e)} = \slashed{\partial}^{(e)} (g) = e.
\]
This completes the proof.

\end{proof}

\section{Spheres on the integer  lattice}\label{SpheresLattice}
We now specialize to the case that the collection $\mb{S}$ comprises the spheres  $\mb{S}_{i, j}$ of radius $1/2$ centered at the points $(0, i, j)$. The singular set $\mc{s}$ in this case comprises the points $\mc{s}^{v}_{i, j} : = (0, i, j + 1/2) :  = \mb{S}_{i, j} \cap \mb{S}_{i, j + 1}$   as well as the points $\mc{s}^{h}_{i, j} : =  (0, i + 1/2, j) = \mb{S}_{i, j} \cap \mb{S}_{i + 1, j}$. We refer to singularities of the form $\mc{s}^{v}_{i, j}$ as \emph{vertical} and to those of the form $\mc{s}^{h}_{i, j}$ as \emph{horizontal}. Correspondingly,  the singular part  $\mc{N}$ of $\mb{S}$ is the union of the two subsets $\mc{N}^{v}$ and $\mc{N}^h$, comprising the necks $\mc{N}^{v}_{i, j} : =  B_{\delta}\left(\mc{s}^v_{i, j} \right) \cap \mb{S} $  and  $\mc{N}^{h}_{i, j} : =  B_{\delta}\left(\mc{s}^h_{i, j} \right) \cap \mb{S} $, respectively. We refer to necks belonging to $\mc{N}^v$ and $\mc{N}^h$ as vertical and horizontal, respectively.   The non-singular part $\mc{S}$ of $\mb{S}$ comprises the components $\mc{S}_{i, j} \subset \mb{S}_{i, j}$, each of which differs from another by a translation. In particular, the  component $\mc{S}_{0, 0}$ lies in the sphere $\mb{S}_{0, 0}$ centered at the origin, and is obtain from $\mb{S}_{0, 0}$ by removing disks of radius $\delta$ from the four points $\left(0, \pm \frac{1}{2}, \pm \frac{1}{2}\right)$. Each component $\mc{S}_{i, j}$ is incident--meaning that it shares a boundary component--with four necks in $\mc{N}$, namely the vertical necks $\mc{N}^v_{i, j}$,  $\mc{N}^v_{i, j + 1}$ and the horizontal necks $ \mc{N}^h_{i, j}$ and $\mc{N}^h_{i + 1, j}$. We refer to $\mc{N}^v_{i, j}$,  $\mc{N}^v_{i, j + 1}$ as the \emph{bottom} and \emph{top} necks incident to $\mc{S}_{i, j}$ and to $ \mc{N}^h_{i, j}$ and $\mc{N}^h_{i + 1, j}$ as the \emph{left} and \emph{right} necks incident to $\mc{S}_{i, j}$.

\subsection{Leaves and branches}
A \emph{leaf} in $\mc{C}$ is a component of $\mc{S}$ along with  the four incident half necks  in $\mc{N}$. Thus, the leaf $\mc{L}_{i, j}$ contains $\mc{S}_{i, j}$ along with the half necks $\mc{N}^{h, +}_{i, j}$, $\mc{N}^{h, -}_{i  + 1, j }$, $\mc{N}^{v, +}_{i, j}$ and $\mc{N}^{v, -}_{i, j + 1}$, to which we refer as the left, right, bottom and top half necks in $\mc{L}_{i, j}$, respectively. We can apply translations to the half necks in each leaf so that $\partial = \partial^*$ for each component $\partial$ of $\partial \mc{S}_{i, j}$.  Given $\xi \in\ul{\xi}$ with $\Lambda_{\xi} = 0$ the canonical  projection $\pi \mc{L}_{i, j}$ of $\mc{L}_{i, j}$ into $\mb{R}^3$  is then  a   smooth surface with boundary $w_{i,j}$  comprising the waists of the four incident half necks in $\mc{L}_{i, j}$. We label the components of $w_{i, j}$ lying in the left, right, bottom and top half necks in $\mc{L}_{i, j}$ as $w_{i, j, l}$, $w_{i, j, r}$, $w_{i, j, b}$ and $w_{i, j, t}$, respectively.

A \emph{branch} in $\mc{C}$ is a collection of leaves belonging to the same column. Thus the $i^{th}$ branch $\mc{B}_i$ in $\mc{C}$ is the collection of leaves $\mc{L}_{i, j}$ where $j$ is fixed. Since the top neck of $\mc{L}_{i, j}$ and the bottom neck of $\mc{L}_{i, j + 1}$ differ by a translation, we can construct a unique piecewise translation of the leaves that vanishes on $\mc{L}_{i, 0}$ and such that  $w_{i, j, t} = w_{i, j + 1, b}$, and we assume throughout that this motion has been applied.  When $\Lambda_{\xi} = 0$,   each branch  is a  smooth surface with boundary  comprising the waists of the horizontal half necks half necks in $\mc{B}$ and each branch depends differentiably on compact subsets of $\mb{R}^3$ on $\xi$.

\subsection{Horizontal and vertical differences. Weighted norms}
In the following, we let $A$ be a banach space, and we let $f = f_{i, j}$ be a family in $A$ indexed by pairs $(i, j)$ of integers.
\begin{definition}\label{Differences}
 The \emph{vertical differences} $f^{\Delta}  $ of $f$ is the family given by $f^{\Delta}_{i, j} = f_{i, j + 1} - f_{i, j }$ . The \emph{horizontal differences} $f^{\vtr}$ of $f$ is the family $f^{\vtr}_{i, j} = f_{i + 1, j} - f_{i , j}$. 
\end{definition}
  We will need the following weighted norms $\| -\|_{\mu}$.
\begin{definition}\label{WeightedNorms}
  Given $\mu \in \mb{R}$, we set $\left\| f\right\|_{\mu} : = \sup_{j} \cosh^{\mu}(j) \left\| f_{i, j}\right\|$.  
\end{definition}
Finally, we will need the \emph{ vertically graded norms}, which we  introduce below:
\begin{definition}\label{GradedNorm}
 We define the vertically graded norm $\left\| -\right\|^{\sim}_{\mu}$ as follows: Given $f \in A$ we set $\left\|f \right\|^{\sim}_{\mu} =\sup_{i}  \sup_j\left(\| f_{i, 0}\| + \cosh^{ \mu} (j)\left\| f^{\Delta}_{i, j}\right\| \right)$ 
\end{definition}
\begin{lemma}\label{Prop:WeightedNorm}
For $\mu > 0$ it holds  that $ \left\| f\right\| \leq^C  \left\| f\right\|^{\sim}_{-\mu}$.
\end{lemma}
\begin{proof}
We have 
\begin{align*}
\|f_{i, j}\| & = \|f_{i, 0}\| + \sum_{k = 1}^{j  - 1} \left\|f^{ \Delta}_{i, j} \right\| \\
& \leq \| f_{i, 0}\| + \sum_{k = 1}^{j - 1} \cosh^{-\mu}(k) \| f \|_{-\mu}\\
& \leq^C \| f\|^{\sim}_{-\mu}
\end{align*}
\end{proof}

\begin{remark}\label{WhatIsMu}
We will  regard $\mu$ as a fixed positive constant throughtout the remainder of this article. We will then choose base flux small  $\tau_0$  small relative to $\mu$ (See, for example Proposition \ref{OperatorinWeightedNorm}).
\end{remark}

\subsection{Distinguished subspaces of parameters}
The type $I$ and type $II$ parameters, as well as their restrictions to the subspaces horizontal and vertical necks, will assume different roles in our construction,  and for this reason it is useful to introduce notation for  the following subspaces of the parameter space $\ul{\xi}$. 

\begin{definition}\label{ParameterSubspaces}
 We let $\ul{e}_I$ denote the subspace of parameters $\ul{\xi}_I$ that vanish on $\mc{N}^h$. We also set  $\ul{e}_{II} = \ul{\xi}_{II}$ and let $\ul{e}$ denote the direct sum of $\ul{e}_{I}$ and $\ul{e}_{II}$. Finally, we let $\ul{v} : = \ul{\xi}_I^h$ the subspace of type $I$ parameters that are supported on $\mc{N}^h$. Thus we have
 \[
 \ul{\xi} = \ul{\xi}_I \oplus \ul{\xi}_{II} = \ul{v} \oplus \ul{e}_I  \oplus \ul{e}_{II} = \ul{v} \oplus \ul{e}.
 \]
\end{definition}

\subsection{$\mc{G}$-invariant perturbations}
 The collection $\mb{S}$ is invariant under the group $\mc{G}$ generated by reflections through the lines $y= i,  i + 1/2$ and $z = j, j + 1/2$, where $i$ ranges over the integers.  The group $\mc{G}®$ also contains the translations by $e_y$ and $e_z$, which act transitively on the horizontal and vertical catenoids, as well as on the spherical domains in $\mc{C}$. This implies that the $\mc{G}$ invariant subspaces of $\ul{\tau}$ and $\ul{\phi}$ are two dimensional, since flux and phase assignments are determined by their values on a single  vertical and horizontal neck, which may be freely prescribed. Since $\mc{G}$ invariance directly implies orthogonality to the obstruction  space $\ul{e}_{*, I}$ we have:
 \begin{proposition}\label{SymmetryPerturbations}
There is $\epsilon > 0$ and $C > 0$ such that: Given $\mc{G}$-invariant $ \xi_I \in \ul{\xi}_{I}$  with $\| \xi_{I}\| < \epsilon$ there is a unique $\mc{G}$ invariant  $e_{II}  = \left(e_{II}\right)_{\xi_I}\in  \ul{e}_{II} $ with $\| e_{II}\| \leq C \| \xi_{I} \|$
and such that 
\[
\Lambda_{\xi_I + e_{II}} = 0.
\]
\end{proposition}
\begin{proof}
$\Lambda$ restricts as a differentiable  map between the $\mc{G}$ invariant subspaces of $\ul{\xi}$ and $\ul{e}_{*}$. Since  the kernel $\ul{e}_{*, I}$ does not survive the symmetries  of $\ul{\mc{G}}$, we have in the $\mc{G}$-invariant setting that $\ul{e}_{II} = \ul{e}$ and thus by Corollary \ref{TypeIIISo}, the linearization $\dot{\Lambda}: \ul{e}_{II} \rightarrow  \ul{e}_{*, II}$ is an isomorphism. The claim then follows directly from the Implicit Function Theorem. 
\end{proof}

The space of $\mc{G}$-invariant type $I$ parameters is four dimensional, with the phase and flux of the vertical and horizonal necks independently prescribable. Let $\ul{\tau}_0$ denote the one dimensional subspace of $\mc{G}$-invariant assignments in $\ul{\xi}_I$ with vanishing phase and with vanishing horizontal flux. We will refer to an element of $\ul{\tau_0}$ as a \emph{base flux assignment}. Proposition \ref{SymmetryPerturbations} then gives a differentiable one parameter family $\tau_0 \in \mb{R} \mapsto e_{\mc{G}, \tau_0}$ of $\mc{G}$-invariant elements of $\ul{e}_{II}$ such that $\Lambda_{\tau_0 + e_{\mc{G}, \tau_0}} = 0$. 

\begin{remark}\label{TotalParamConv}
Throughout the remainder of this article, we will regard the base flux assignment  $\tau_0 \in \ul{\tau}$ as fixed and  small subject to various considerations. We  will  suppress from our notation the dependence of the surfaces $\mc{C}_{\xi}$ on $\tau_0$ as follows: Set $e_0 = e_{\mc{G}, \tau_0}$  and $\xi_{0} :  = \tau_0 + e_0$. We will then write $\Lambda_{\xi}$ to mean   $\Lambda_{\xi + \xi_0}$, unless additional clarity is required. For a given parameter value $\xi \in \ul{\xi}$, we will refer to $\xi + \xi_0$ as the total parameter value corresponding to $\xi$ and denote it by $\xi_{\text{total}}$ and we refer to $\xi_0$ as the background parameter of the construction.
\end{remark}

Since the horizontal flux assignments vanish when   $\xi = 0$, the branches of $\mc{C}$ are smooth surfaces without boundary, invariant under reflections through the lines $y = j$  and $y = j + 1/2$ where $j$ ranges over the integers. It then follows easily that each branch is a surface of rotation and thus a Delaunay surface.  Though this is a standard fact, we include a short proof below:

\subsection{Branches  at $\xi = 0$ are Delaunay surfaces}
Since the horizontal flux assignments vanish when   $\xi = 0$, the branches of $\mc{C}$ are smooth constant mean curvature surfaces without boundary and two ends, and are thus Delaunay surfaces.  Though this is a standard fact, we include a short proof below:

\begin{proposition}\label{RotSymBranches}
Each branch at $\xi = 0$ is a smooth CMC surface of rotation and thus a Delaunay surface.
\end{proposition}

We will need in the proof of Proposition \ref{RotSymBranches} the following fact, which is well known and whose proof we therefore omit:
\begin{lemma}\label{CatJacFields}
The space of bounded jacobi fields on the catenoid is precisely the space of translational jacobi fields. 
\end{lemma}
In other words, the space of bounded jacobi fields on the catenoid is spanned by the components of the unit normal. 
 \begin{proof}[Proof of Proposition \ref{RotSymBranches}]
 Suppose that a branch $\mc{B}$ is not rotationally symmetric. Then the rotational  vector field $x e_y  -  ye_x$ has non-trivial normal part, which we denote by $u$.  Observe that $u$ is a  periodic jacobi field and moreover the symmetries of $\mc{B}$ imply that $u$ is odd with respect to coordinate plane reflections.  We normalize $u$ so that $\sup u = 1$.   Now, take a sequence $\tau_{0, j}$ of base flux assignments tending to  $0$ and consider the corresponding sequences  $\mc{B}_j$ of branches and  $u_{j}$ of periodic jacobi fields  on $\mc{B}_j$. Observe that as $\tau_{0, j} \rightarrow 0$ the branch $\mc{B}_j$ converges to a  limit $\mc{B}_{0}$, which is a collection of spheres of radius $1/2$ touching at the north and south poles--that is, the points where the unit normal coincides with $\pm e_z$.  Pick a sequence $p_j$ of points realizing the supremum of $u_j$, so $1 = u_j (p_j)$. There are then several possibilities. If $p_j$ remains in a fixed compact subset away from the poles, then $u_j$ converges to a non-trivial  bounded jacobi field on the sphere away from the north and south poles. Standard removable singularity theory then implies that it extends smoothy across the puncture to a non-trivial jacobi field on the sphere, and is thus a  translational jacobi field. However, this is a contradiction, since we assumed that $u$ is orthogonal to the translational jacobi fields over each fundamental domain. Thus, we must have that $p_j$ converges to a pole, which without loss of generality we can take to be the origin. After applying a small translation in the $e_z$ direction, we can assume that $p_j$ lies in the plane $\{z = 0 \}$. The rescaled Delaunay surface $\tilde{\mc{B}}_i = \mc{B}_i/ |p_i|$ then converges to a limit $\tilde{\mc{B}}$ that is either a flat plane or else the standard catenoid. In both cases, define the function $\tilde{u}_j$ on $\tilde{\mc{B}}_j$ by $\tilde{u}_j (p) = u(|p_j| p)$. Then $\tilde{u}_j$ is a jacobi field on  $\tilde{\mc{B}}_j$ that achieves its supremum, which is equal to $1$, on the unit circle. The sequence is then uniformly bounded in $C^k$ and converges to a nontrivial   bounded jacobi field  $u$ on $\wt{\mc{B}}$. In the case that $\wt{\mc{B}}$ is the punctured plane, $u$ extends smoothly across the origin  to a bounded harmonic function  on the plane, and is thus constant. However, since $u$ is odd with respect to coordinate plane reflections, this is a contradiction. On the other hand, if the limit $\wt{\mc{B}}$ is the standard catenoid, then we obtain a contradiction to Lemma \ref{CatJacFields} by again considering the symmetries of $u$. This completes the proof.

  \end{proof}

\section{The defect operator on Delaunay surfaces}\label{DefectOnDelaunay}
In this section, we record the effects of the type $I$ parameters on the kernel $\ul{e}_{*, I}$. We will need that the variations in the type $I$ vertical parameters generate the kernel in order to apply the implicit function theorem, and we need to understand the projections of type $I$ horizontal parameters in proving embeddedness of the solutions later.

\subsection{Type II projections}
We now record the fact that the linearized operator $\dot{\Lambda}$ preserves the  type II spaces. This  is the case for the operator $\dot{\Lambda}_0$--the linearization of $\Lambda$ at vanishing total parameter value $\xi_{\text{total}} = 0$ (Recall Remark \ref{TotalParamConv}), and it is a convenient fact that this carries of over to the linearization of $\Lambda$ at $\xi_0$. As a corollary, we can use a pertubation argument to show that it is an isomorphism from the space $\ul{e}_{II}$ to $\ul{e}_{*, II}$

\begin{lemma} \label{DotLambdaRestricts}
The operator $\dot{\Lambda}$ restricts as a  mapping from $\ul{e}_{II}$ into $\ul{e}_{*, II}$.
\end{lemma}
\begin{proof}
Pick $\dot{e} =  (\dot{\rho}, \dot{f}) \in \ul{e}_{II}$. We will show that $\dot{\Lambda}\left( \dot{e}\right) $ is perpendicular to $\mc{k}^{y}_{i, j}$ and $\mc{k}^z_{i, j}$ for arbitrary $i$ and $j$ where here $\mc{k}^{y}_{i, j}$ and $\mc{k}^z_{i, j}$ denote the restrictions of the translational jacobi field  $\mc{k}^{y}$ and $\mc{k}^z$ to $\mc{S}_{i, j}$. Consider the leaf $\mc{L} = \mc{L}_{i, j}$. We will assume throughout that coordinate axes are arranged so that $\mc{L}_{i, j}$ is invariant under reflections through the coordinate planes, and we will suppress the index $(i, j)$ from our notation since we are regarding it as fixed. Recall that the boundary of $\mc{L}$ comprises $\partial \mc{S} = \partial \mc{S}_{i, j}$ as well as $(\partial \mc{S})^*$, along with the waists $w$ of the half necks in $\mc{L}$. In the following, we will interpret the pairing $*$ as a pairing on $\partial \mc{L}$ that acts on $w$ as the identity.  First, we will need a few preparations that describe geometric data on rotationally symmetric surfaces.

Since each branch is a rotationally symmetric surface, it can be parametrized by $F(x, s) = r e_r (x) + h e_z$, where $e_r(x) = \cos(x) e_x + \sin(x) e_y$, and where $r$ and $h$ are functions of $s$ only. We will assume that  $s$ is an arc length parameter for the profile curve so that $h'^2 + r'^2 = 1$.  The components $N_x$, $N_y$ and $N_z$ of the unit normal are then given respectively by $h' \cos(x)$, $h' \sin(x)$ and $- r'$. In particular, since $r'$ doesn't vanish on the top and bottom  boundary components $\partial_{t}$ and $\partial_b$ of $\mc{S}$, and since $x= r\cos(\theta)$, $y = r\sin(\theta)$ and $z = 1$, we have that the components $N_x$,  $N_y$ and $N_z$ of the unit normal are uniformly comparable to $\mc{k}^x$, $\mc{k}^y$ and $\mc{k}^z$, respectively, and thus we can compute the $L^2$ projections of $\dot{\Lambda} (\dot{e})$  onto  $\ul{e}_{II}$  up to positive constants by computing the projections  onto the $x$, $y$, $z$ components $N_x$, $N_y$ and $N_z$ of the unit normal $N$ to $\mc{S}$. Observe that, along the waists $w_t$ and $w_b$, the symmetries of the branches imply that $r' = 0$ and thus $N_z = 0$. Differentiating the equality $r'^2 + h'^2 = 1$ gives
$r' r'' + h' h'' = 0$.  In particular, along the waist we have $h' = 1$ and thus $h'' = 0$. The expressions for the components $N_x$, $N_y$ and $N_z$ o the unit normal  then give $\left( N_y \right)_{\eta} = \pm N_y' = h'' \sin(x) = 0$ and $N_{z, \eta} = \pm N'_{z} = \pm r''$. Now, let $u$ denote the normal variation field in $\mc{L}$ generated by $\dot{e}$. We then have $ \dot{\Lambda}  (\dot{e}) = \left(u_{\eta}\right)^{(e)}$. Since $u$ is a jacobi field on $\mc{L}$, integrating by parts gives
\[
0 = \int_{\partial \mc{L}} u_\eta N - u N_{\eta}.
\]
Applying the pairing $*$ and using that $N$ is even and $N_{\eta}$ is odd, we get
\begin{align} \label{PairingApplied}
2\int_{\partial \mc{S} + \partial \mc{S}^*}\left(  \dot{\Lambda} \right)  (\dot{e}) N + 2\int_{w} u_{\eta} N - 2 \int_{w} u N_{\eta} = 0.
\end{align}
We consider first the case that $N = N_y$ . Then  the last term on the left hand side above vanishes and we are left with 
\begin{align}\label{NProjection1}
\int_{\partial \mc{S} + \partial \mc{S}^*}\left(  \dot{\Lambda} \right)  (\dot{e}) N_y  = - \int_{w} u_{\eta} N_y.
\end{align}
In the case that $\dot{e} = (0, \dot{f})$, so that $\dot{\rho} = 0$, the variation field $u$ is given by $u = \dot{\mc{h}}_f$,  which is normalized. Recall that this means that $\dot{\mc{h}}_{f}$ and $\left(\dot{\mc{h}}_{f}\right)_{\eta}$ are orthogonal to the space of lower modes on the waists $w_{t}$ and $w_{b}$ of the top and bottom half necks.  Since  $N_y$  restricts to $w$ as a lower mode, the second term on the right hand side above vanishes and we have $\int_{\partial \mc{S} + \partial \mc{S}^*}\left(  \dot{\Lambda}  (\dot{e})\right)  N_y = 0$, and thus $\int_{\partial \mc{S}}\left(  \dot{\Lambda}   (\dot{e}) \right)\mc{k}^{y} = 0$.  In the case that $\dot{e} = (\dot{\rho}, 0)$, then the the restriction of the  variation field $u$ to each half neck in $\mc{L}$ is a translational jacobi field and thus belongs to  the span of $N_y$ and $N_z$. Suppose first that the restriction of $u$ to $\mc{H}_{t} : = \mc{N}^{v, -}_{i, j}$ is equal to $N_y$. Then   $u_{, \eta} = N'_{y} = 0 $ on $w_{t}$ and thus the right hand side of (\ref{NProjection1}) vanishes over $w_{t}$. Suppose now that the restriction of $u$ to $\mc{H}_t$ is $N_z$. Then $u_{\eta}= \pm r'' $ and thus
\[
\int_{w_t} u_{\eta} N_{y} = \pm r'' h'\int_{0}^{2\pi}  \sin(x)  dx = 0.
\]
A similar analysis can be applied to $w_b$, which gives in all cases that $\int_{\partial \mc{S}}\left(  \dot{\Lambda}   (\dot{e}) \right) \mc{k}^y = 0$.

We now consider the case that  $N = N_z$. We have $N_z = 0$ along $w$ and thus the second term in (\ref{PairingApplied}) vanishes and we have
\begin{align}\label{NProjection2}
\int_{\partial \mc{S} + \partial \mc{S}^*}\left(  \dot{\Lambda}  (\dot{e}) \right) N_z  =  2 \int_{w} u N_{z, \eta} = 2 \pm r'' \int_{w} u.
\end{align}
In the case that $\dot{e} = (0, \dot{f})$, then as before $u$ is normalized since the restriction of $u$ to $\mc{N}$ is $\dot{\mc{h}}_{\mc{N}, f}$. Since $N_{z, \eta} = \pm r''$ is constant and thus a lower mode, the last term above vanishes. When $\dot{e} = \left(\dot{\rho}, 0 \right)$, then again the restriction of $u$ to each half neck in $\mc{L}$ is in the span of   $N_y$ and $N_z$. Suppose first that the restriction of $u$ to $\mc{H}_t$ is $N_z$. Then $u$ vanishes on $w_{t}$ and thus $\left(  \dot{\Lambda}  (\dot{e}) \right)$ and the integral on the right hand side of (\ref{NProjection2}) vanishes over $w_t$.  Assume now that  the restriction of $u$ to $\mc{H}_t$ is $N_y$. Then $u \cong \sin(x)$ on $w_t$ and again the integral on the right hand side of (\ref{NProjection2}) vanishes over $w_t$. Applying the same analysis to $w_b$, we conclude that  (\ref{NProjection2}) vanishes in all cases and thus $\dot{\Lambda} (\dot{e})$ is in $\ul{e}_{*, II}$ as claimed. 

\end{proof}

\begin{corollary}\label{NoKernelISo}
$\dot{\Lambda}: \ul{e}_{II} \rightarrow \ul{e}_{*, II}$ is an isomorphism in the norms $\| -\|_{\mu}$.
\end{corollary}
\begin{proof}
This is a relatively standard perturbation argument. The point is that the operator $\dot{\Lambda}$ is mostly local, in the sense that $\dot{\Lambda}_{i, j}$ depends only on $\dot{e}_{i, j}$ and $\dot{e}_{i, j  - 1}$. The uniform control on the rate of exponential growth implies the nearby arguments are comparable and thus contribute error terms of roughly the same size. Pick $\dot{e}_* \in \ul{e}_{*, II}$. Then there is $\dot{e}_0 \in \ul{e}_{II}$ such that $\dot{\Lambda}_0 (e_0) = \dot{e}_*$. We then have $\dot{\Lambda} (\dot{e}) = \dot{e}_{*} + \dot{e}_{*, 1}$, where $e*_1$ satisfies the estimate
\begin{align*}
\left\|\left( \dot{e}_{*, 1} \right)_{i, j}\right\|  & \leq^C \left\|\dot{\Lambda}  - \dot{\Lambda}_0\right\|\left(\left\| \dot{e}_{i, j} \right\| + (\left\| \dot{e}_{i, j - 1} \right\|\right) \\
& \leq^C \tau_0 \left(\cosh^{\mu} (j) + \cosh^{\mu} (j - 1)\right)\\
& \leq^C \tau_0 \cosh^{\mu}(j).
\end{align*}
Thus, $e_{*, 1}$ is in $\ul{e}_{*, II}$ and satisfies the estimate $\| e_{*, 1}\| \leq 1/2$ for $\tau_0$ sufficiently small. The process can then be iterated to obtain an exact solution.  
\end{proof}

\subsection{Type I projections}
In this section, we record the projections of the variations generated by type $I$ parameters onto the kernel elements $\ul{e}_{*, I}$. Recall (Corollary \ref{TypeIIISo}) that the type II parameters generate the orthogonal complement $\ul{e}_{*, II}$ of $\ul{e}_{*, I}$. Here we show that the type I parameters generate the co-kernel of the type II variations. Recall that  $\ul{e}_{*, I}$ is defined to be the space $\mc{tr}\left(\wc{\ul{\mc{k}}} \right)$ of traces of elements of $\ul{\wc{\mc{k}}}$, where here $\ul{\mc{k}}$ denotes the space of jacobi fields on $\mb{S}$ and $\wc{\ul{\mc{k}}}$ the space of lifts of elements of $\ul{\mc{k}}$ to $\mc{C}$ under the canonical projection $\pi: \mc{C} \rightarrow \mb{S}$.  The space of jacobi fields on a sphere is generated by translations, and is spanned by the components of the unit normal. If $\mb{S}_{i,j}$ is the sphere in $\mb{S}$ centered at the point $(0, i, j)$, we let $\mc{k}^{x}_{i, j}$,  $\mc{k}^{y}_{i, j}$ and $\mc{k}^{z}_{i, j}$ denote the jacobi field on $\mb{S}_{i, j}$ generated by translations in the direction $e_x$, $e_y$ and $e_z$, respectively. Thus,  $\mc{k}^{x}_{i, j}$,  $\mc{k}^{y}_{i, j}$ and $\mc{k}^{z}_{i, j}$ are the $x$ $y$ and $z$ components of the unit normal of $\mb{S}_{i, j}$ and the space $\ul{\mc{k}}$ is spanned by these functions. In the following, we will  not distinguish a function $\mc{k} \in \ul{\mc{k}}$ with the trace of its lift $\mc{tr} \left( \wc{\mc{k}}\right)$. Thus, the space $\ul{e}_{*, I}$ is spanned by the functions $\mc{k}^{x}_{i, j}$,  $\mc{k}^{y}_{i, j}$ and $\mc{k}^{z}_{i, j}$.   The symmetries we have imposed on our construction imply that the error terms are necessarily orthogonal to $\mc{k}^{x}_{i, j}$, and thus we record here only the projections onto   $\mc{k}^{y}_{i, j}$ and $\mc{k}^{z}_{i, j}$.

\begin{lemma}\label{TauProjections}
It holds that 
\[
\left\langle \frac{\partial \Lambda}{\partial \tau}\left(\dot{\tau} \right), \mc{k}^{y}_{i, j} \right\rangle = - c_{\delta} \dot{\tau}^{h, \vtr}_{i, j}
\]
and 
\[
\left\langle \frac{\partial \Lambda}{\partial \tau}\left(\dot{\tau} \right), \mc{k}^z_{i, j} \right\rangle = - c_{\delta} \dot{\tau}^{v, \Delta}_{i, j}
\]
where $c_\delta =  \pi \delta \sqrt{1 - \delta^2}\left(\frac{1}{2} - \delta^2 \right) $.
\end{lemma}
\begin{proof}
Observe that symmetry considerations give that 
\[
\left\langle \frac{\partial \Lambda}{\partial \tau}\left(\dot{\tau}^{v} \right), \mc{k}^{y}_{i, j} \right\rangle  = 0,
\]
and thus
\[
\left\langle \frac{\partial \Lambda}{\partial \tau}\left(\dot{\tau} \right), \mc{k}^{y}_{i, j} \right\rangle  = \left\langle \frac{\partial \Lambda}{\partial \tau}\left(\dot{\tau}^{h} \right), \mc{k}^{y}_{i, j} \right\rangle.
\]
Let $\partial_{l}$ and $\partial_{r}$ denote the left and right boundaray components of  $\partial \mc{S}_{ij}$,  respectively. Since  at $\xi = 0$ we have 
\[
\left. \frac{\partial \lambda}{\partial \tau} (\dot{\tau}) \right|_{\partial \mc{N}^{h}_{i, j}} = -\dot{\tau}^{h}_{i, j},
\]
we have 
\[
\left. \dot{\Lambda}\left( \dot{\tau}\right) \right|_{\partial_l} = - \frac{ \dot{\tau}^{h}_{i, j}}{2}, \quad \left. \dot{\Lambda}\left( \dot{\tau}\right) \right|_{\partial_r} = - \frac{ \dot{\tau}^{h}_{i + 1, j}}{2}
\]
The boundary component $\partial_{l}$ is the circle of radius $r  = \delta \sqrt{1  - \delta^2}$ in the plane $x_{ij} = - 1/2 + \delta^2$ centered at the point $(-1/2 +  \delta^2, 0, 0)$. Thus on $\partial_{l}$ we have $\mc{k}^{y}_{i, j} = - 1/2 + \delta^{2}$ which gives
\begin{align*}
\int_{\partial_l}  \dot{\Lambda}\left(\dot{\tau} \right) \mc{k}^{y}_{i, j} & = \dot{\tau}^{h}_{i, j}\int_{\partial_{l}} \left( -\frac{1}{2} \right)\left( -1/2 + \delta^2\right) \\
& = \pi \delta \sqrt{1 - \delta^2}\left(\frac{1}{2} - \delta^2 \right) \dot{\tau}^h_{i,j}.
\end{align*}
Similarly,  $\partial_r$ is the circle of radius $r  = \delta \sqrt{1  - \delta^2}$ in the plane $x_{ij} = + 1/2 - \delta^2$ centered at the point $(1/2 - \delta^2, 0, 0)$, and thus we get
\begin{align*}
\int_{\partial_r} \frac{\partial \Lambda}{\partial \tau}\left(\dot{\tau} \right) \mc{k}^{y}_{i, j} & = \dot{\tau}^{h}_{i + 1, j}\int_{\partial_{r}} \left( -\frac{1}{2} \right)\left( \frac{1}{2} - \delta^2\right) \\
& =- \left( \pi \delta \sqrt{1 - \delta^2}\left(\frac{1}{2} - \delta^2 \right)\right) \dot{\tau}^h_{i,j}.
\end{align*}
The proof of the remaining estimate proceeds identically to the first, and we thus omit it. This completes the proof. 
\end{proof}

\begin{lemma}\label{Projections}
It holds that 
\[
\left\langle \frac{\partial \Lambda}{\partial \phi} (\dot{\phi}), \mc{k}^{y}_{i, j} \right\rangle =  c_0 \phi^{v, \Delta}_{i, j}.
\]
For a constant $c_0 > 0$ depending on $\tau_0$.
\end{lemma}
 \begin{proof}

The leaf $\mc{L}_{i, j}$ comprises $\mc{S}_{i, j}$ as well as the four incident half necks $\mc{N}^{h, +}_{i, j}$, $\mc{N}^{h, -}_{i +1 , j}$ $\mc{N}^{v, +}_{i, j}$ and $\mc{N}^{v, -}_{i, j + 1}$.  We will assume throughout that the coordinate axes have been chosen to coincide with the center $(0,  i, j)$ of $\mb{S}_{i, j}$. We will also let $N_x$, $N_y$ and $N_z$ denote components of the unit normal to $\mc{C} = \mc{C}_{\tau_0 + e_{\mc{G},  \tau_0}}$. Observe the symmetries of $\mc{L}_{i, j}$ give that  along $\partial_{t}$ and $\partial_b$, $\mc{k}^{y}_{i, j}$ and $N_y$ are proportional up to a positive constant, so $\mc{k}^{y}_{i, j} = \alpha N_y$  for some $\alpha > 0$ on $\partial_b$ and $\partial_t$. Thus, in order to estimate the $L_2$-projection $\langle \frac{\partial \Lambda}{\partial \phi} (\dot{\phi}), \mc{k}^{y}_{i, j} \rangle$ it suffices to estimate $ \langle \frac{\partial \Lambda}{\partial \phi} (\dot{\phi}), N_{y} \rangle_{\partial \mc{S}_{i, j}}$.

The parameter $\phi$ rotates $\partial_{t}$ about the origin through the positive angle $\phi^{v}_{i, j + 1}$ and thus on $\mc{N}^{v, -}_{i, j + 1}$ the parameter $\phi$ also acts as the same rotation. Let $u$ denote the normal variation field generated by $\phi^{v}_{i, j + 1}$ at $\phi  = 0$. Then 
\[
\left. \dot{\Lambda} \right|_{\partial \mc{S}_{i, j}} = \left(u_{\eta}\right)^{e}.
\]
Integrating over the leaf $\mc{L}_{i, j}$ we have
\[
0 = \int_{\mc{L}_{i, j}} N_yL u = \int_{\partial \mc{L}_{i, j}} u_{\eta} N_{y} - u N_{y, \eta},
\]
where above $L$ denotes the stability operator of $\mc{C}$.
Since $N_{y}$ is smooth on $\pi \mc{L}$ it is Cauchy matched and we have $\left. N_{y, \eta} \right|_{\partial_t} = - \left. N_{y, \eta} \right|_{\partial_t^*}$. This then gives
\[
\int_{\partial \mc{S}_{i, j}}\dot{\Lambda} N_y + \int_{w_{i,j}} u_{\eta} N_{y} - u N_{y, \eta} = 0. 
\]
Thus, it remains to compute the second integral in the left hand side above. 
On $\mc{N}^v_{i, j + 1}$,  $u$ is the rotational field given by
\[
u = y N_z - z N_y.
\]

Let $w_{t}$ denote the top boundary component of $\pi \mc{L}_{i, j}$. Thus, $w_t$  is the waist of the top  half neck in  $\mc{L}_{i, j}$. Then $w_{t}$ is a circle of radius $\delta$, invariant under rotations about the $z$-axis. Let $\theta$ denote the standard cylindrical coordinate on $w_t$. On $w_{t}$, we have $N_{z} = 0$ and $N_{y, \eta} = 0$, and $z_{\eta} = 1$. Moreover, we have $N_y = \sin(\theta)$ and thus $y = r_0 N_y  = r_0 \sin(\theta)$, where $r_0$ denotes the radius of  $w_t$. Finally, since the mean curvature of the Delaunay end is $1$ with respect to the outward unit normal, we have that $N_{z, \eta} = \left(1 - \frac{1}{r_0}\right)$. Using the expression for $u$ above we then have
\begin{align*}
u_{\eta} & = \left( y N_z - z N_y. \right)_{\eta} \\
&  =  y N_{z, \eta} - z_{\eta} N_{y} \\
 &= y \left(1 - \frac{1}{r_0}  \right) - N_y \\
 & = \sin(\theta) (r_0 - 1) - \sin(\theta) \\
 &=  (r_0- 2 ) \sin(\theta)
\end{align*}
Thus we have
\begin{align*}
 \int_{w_t} u_{\eta} N_{y} - u N_{y, \eta} & = \int_{w_t} u_{\eta} N_y ds \\
 & =   (r_0 - 2)\int_{0}^{2\pi} \sin^2(\theta) d\theta \\
 & = \frac{r_0 - 2}{2}
\end{align*}
Thus, we have
\[
\left\langle \frac{\partial \Lambda}{\partial \phi} (\dot{\phi}), \mc{k}^{y}_{i, j} \right\rangle =  \frac{2 - r_0}{2} \phi^{v, \Delta}_{i, j}
\]
as claimed. This completes the proof. 
\end{proof}

\section{Symmetry breaking perturbations} \label{SymBreaPert}

In this section we record the mapping properties of the operator $\Lambda$ as a  map from the space $\ul{\xi}$ into $\ul{e}_{*}$ in the vertically graded norms with exponentially decaying weights. The main results of the section are Proposition \ref{OperatorinWeightedNorm}, which says that $\Lambda$ is defined and bounded in $C^1$ on a ball about the origin and Proposition \ref{VerticalVariations}, which says the linearization is an isomorphism from $\ul{e}$ onto $\ul{e}_*$. As a direct consequence, in Corollary \ref{SymmetryBroken} we construct a differentiable map $v \mapsto e_v$ parametrizing the zeroes of $\Lambda$ near $\xi = 0$, so $\Lambda_{v + e_v} = 0$.

\begin{proposition}\label{OperatorinWeightedNorm}
There is $\epsilon > 0$ such that $\Lambda: \ul{\xi} \rightarrow \ul{e}_{*}$ is defined and differentiable on  the ball $B_{\epsilon} \left(0 ; \ul{\xi}\right)$ in the norm vertically graded norm $\| -\|^{\sim}_{-\mu}$.
\end{proposition}
\begin{proof}
By Lemma \ref{Prop:WeightedNorm} we have $\| \xi\| \leq^C \| \xi\|^{\sim}_{-\mu}$. Moreover, since $\Lambda^{\Delta} (\xi)  = \Lambda (\xi^{\Delta})$, we have
\[
\left\|\Lambda (\xi)\right\|^{\sim}_{-\mu} \leq^C \left\|\xi\right\|^{\sim}_{-\mu},
\]
so that $\Lambda$ is bounded in $C^0$ on the ball of radius $\epsilon$ about zero in $\ul{\xi}$ in the form $\| -\|^{\sim}_{-  \mu}$.
The derivative estimate follows similarly. 
\end{proof}
The following proposition records an a-priori estimate for the operator $\Lambda$ in the norms $\| -\|^{\sim}_{- \mu}$
\begin{proposition}\label{SimAprioriEstimate}
There is $\epsilon > 0$ such that: For each $\tau < \epsilon$, there is a constant  $C > 0$ such that:  Given  $e \in \ul{e}$ it holds that 
\[
\left\|e\right\|^{\sim}_{-\mu} \leq^C \left\| \dot{\Lambda}(e) \right\|^{\sim}_{-\mu}.
\]
\end{proposition}
\begin{proof}
We will first show that 
\begin{align} \label{DifferenceEngine}
\left\|e^{\Delta}\right\|_{-\mu} \leq^C \left\| \dot{\Lambda}^{\Delta}(e) \right\|_{-\mu}.
\end{align}
 If this doesn't hold, then  for arbitrarily small base flux $\tau_0$, there  is a sequence   $e_k$  with $\left\|\dot{\Lambda}^{\Delta} (e_k)\right\|_{-\mu} \rightarrow 0 $ and $\left\|e_k^{ \Delta}\right\|_{- \mu} = 1$. Pick integers $i_k$ and $j_k$ such that 
 \[
 1 = \left\| e^{\Delta}_k\right\|_{-\mu} = \cosh^{\mu} (j_k) \left\| \left(e^{\Delta}_k\right)_{(i_k, j_k)}\right\|.
 \]
 Let $\ol{e}_k$ denote the element of $\ul{e}$ such that $\left(\ol{e}_{k}\right)_{i, j} = \left(e_{k} \right)_{i_k, j_k} $
 for arbitrary $i$ and $j$ and thus $\ol{e}_{k}^{ \Delta}  = 0$.  Without loss of generality we can assume that $i_k = 0$ and that $j_{k} \geq 0$.   We  will suppress $i$ from the notation throughout the rest of the proof.  We then define the renormalized sequence $\tilde{e}_{k}$  as follows:
\[
\left(\tilde{e}_{k} \right)_{j} : = \cosh^{\mu} (j_k) \left(\left(e_{k} \right)_{ j_k + j} - \left(\ol{e}_{k} \right)_{ j + j_k}\right).
\]
  After passing to a  subsequence we can assume that the sequence $j_k$  is convergent with limit $j_\infty$  in the extended integers $\mb{Z} \cup \{ \infty \}$.   Observe that by construction we have $\left(\tilde{e}_{k} \right)_{0} = 0$ and that 
\begin{align*}
\left\|\left(\tilde{e}^{ \Delta}_{k} \right)_{j} \right\| &= \cosh^{ \mu} (j_k) \left\| \left(e^{\Delta}_{k}\right)_{j + j_k}\right\| \\
& \leq \cosh^{- \mu}_{j_k} (j), 
\end{align*}
where above we have used that $\left\| e_k^{\Delta}\right\|_{-\mu} = 1$, and by definition we have $\left\|\left(\tilde{e}^{ \Delta}_{k} \right)_{0} \right\|= 1$.  After passing to a subsequence, we have that $\tilde{e}_{k}$ converges to a non-trivial limit $e \in \ul{e}$. The limit $e$ then satisfies $\left(e \right)_{0} = 0$ and the estimate $\left\| \left(e^{ \Delta}\right)\right\| \leq \cosh_{j_\infty} (j)$. We have:
\begin{align*}
\left\{\dot{\Lambda} (\tilde{e}_{k}) \right\}_{j}
& = \cosh^{\mu} (j_k) \left\{\dot{\Lambda} \left(e_{k} \right)\right\}_{j + j_k} - \cosh^{\mu} (j_k) \left\{\dot{\Lambda}  \left(\ol{e}_{k} \right)\right\}_{j + j_k}.\\
\end{align*}
Since $\ol{e}^{\Delta}_k = 0$ we have  $\left( \dot{\Lambda}  \left(\ol{e}_{k} \right) \right)^{\Delta}= 0$ and thus
\begin{align*}
\left\|\left\{\left(\dot{\Lambda} (\tilde{e}_{k}) \right)^{\Delta} \right\}_{j} \right\| & \leq \cosh^{\mu}(j_k) \left\|\dot{\Lambda}  (e_{k}) \right\|_{-\mu} \cosh^{-\mu} (j + j_k)\\
& \leq \cosh^{-\mu}_{j_k} (j) \left\| \dot{\Lambda}(e_{k}) \right\|_{-\mu}.
\end{align*}
Thus, $\left\{\left(\dot{\Lambda} (\tilde{e}_{k})\right)^{\Delta} \right\}_{j} \rightarrow 0$ as $k \rightarrow \infty$. Taking the limit as $k \rightarrow \infty$ gives
\[
\left(\dot{\Lambda} (e)\right)^{\Delta} = 0. 
\]
Let $e_I$ and $e_{II}$ denote the projections of $e$ into the spaces $\ul{e}_{I}$ and $\ul{e}_{II}$, respectively.
By Lemmas \ref{TauProjections} and \ref{Projections} and \ref{DotLambdaRestricts} we have that 
\[
\left\| \left\{\left(\dot{\Lambda} (\dot{e})\right)_{j }\right\}_{I} \right\| \leq^C \left\| \left\{e^{\Delta}\right\}_{j}\right\| \leq \cosh_{j_\infty}(j).
\]
Since $\left\{\dot{\Lambda} (e) \right\}_{j}$ is independent of $j$, the decay estimate above forces $ \left(\dot{\Lambda} (e)\right)_{I} $ to vanish. Applying  Lemmas \ref{TauProjections} and \ref{Projections} and \ref{DotLambdaRestricts}  again and the fact that $\left\{e_I \right\}_{0} = 0$ gives that $e_I \equiv 0$ and thus $e$ is in $\ul{e}_{II}$. 

Now, take a sequence $\tau_{0, n}$ of base flux assignments tending to $0$, and let $e_n$ denote the corresponding solutions constructed above. The solutions then converge to a non-trivial limit $e$ satisfying $\dot{\Lambda}_0^{\Delta} (e) = 0$, where here $\dot{\Lambda}_0$ denotes the linearization of $\Lambda$ at $\xi_{total} = 0$.  Let $u$ denote the normal variation field generated by $e$. Then $u$ is a jacobi field  on the branch $\mc{B}_0$  containing $\mc{S}_{0, 0}$. Since the background flux $\tau_0$  is now zero, $\mc{B}_0$ is a collection of spheres, and thus $u$ is a collection of translational jacobi fields on the spheres $\mb{S}_j$ comprising the leaves of $\mc{B}_0$. Since $e_0 =0$, we have that $u$ vanishes on $\mb{S}_0$, from which it follows that $u \equiv 0$ and thus $e = 0$.   This establishes the estimate (\ref{DifferenceEngine}). 

Now, suppose that for small $\tau_0$, there is a sequence $e_k$  with $\left\|e_k \right\|^{\sim}_{-\mu} = 1$ and $\left\| \dot{\Lambda} (e_k)\right\|^{\sim}_{-\mu} \rightarrow 0$. By the estimate (\ref{DifferenceEngine}), we have $\left\| e^{\Delta}_{k}\right\|_{- \mu} \rightarrow 0$ and thus  the sequence $e_k$ converges to a non-trivial limit $e \in \ul{e}$ with $e^{\Delta} = 0$ and satisfying $\dot{\Lambda} (e) = 0$. Since $e^{\Delta} = 0$, the symmetries imply that  $\left(\dot{e}_I\right)_{0} = 0$ and thus $e_I = 0$. Taking the base flux $\tau_0 \rightarrow 0 $ as above we obtain a contradiction as above. This completes the proof. 
\end{proof}

In the following lemma we show that the linear operator $\dot{\Lambda}$ is an isomorphism in the graded norms with exponentially growing weights. This  is needed for a perturbation argument in the proof of Lemma \ref{VerticalVariations}, which records that the operator is an isomorphism in the vertically graded norms with exponentially decaying weight.

In order to simplify our presentation, we will assume throughout that the collection $\mc{C}$ derived from $\mb{S}$, and the corresponding parameter spaces, are invariant under reflections through the plane $\{y = 0 \}$. Thus, given $\xi \in \ul{\xi}$, the entries of $\xi_{i, j}$ of $\xi$ are determined by their values for $j \geq 0$.  Observe additionally that the symmetries imply that if $e_* \in \ul{e}_*$ then $e_{*, i, 0}$ is orthogonal to translational jacobi field $\mc{k}^{z}$. Also, given $e_{I} \in \ul{e}_{I}$, we have $e_{I, i, j} = - e_{I, i, - j + 1}$. Since $e_{*, i, 0}$ is a-priori orthogonal to $\mc{k}^z$, we will assume throughout that $\tau^{v}_{i, 1} = \tau^{v}_{i, 0} = 0$.

\begin{lemma}\label{GrowingIso}
The mapping $\dot{\Lambda}: \ul{e} \rightarrow \ul{e}_{*}$ is a bounded  isomorphism in the weighted norms $\| -\|_{\mu}$ for $\mu > 0$.
\end{lemma}
\begin{proof}
Clearly, we have that $\dot{\Lambda}$ is bounded in the norms $\| -\|_{\mu}$ for arbitrary $\mu$. To see that it has a bounded  inverse in the norms $\| -\|_{\mu}$ for $\mu > 0$, pick $e_* \in \ul{e}_*$ with $\| e_*\|_{\mu} = 1$. By Lemma's \ref{TauProjections} and \ref{Projections} there is a unique $e_I \in \ul{e}_I$  such that $\dot{\Lambda} (e_I) + e_{*} \in \ul{e}_{II}$. Moreover, we have the estimate  $\left\|\left(e^{\Delta}_{I} \right)_{ i, j} \right\| \cong \left\|e_{* i, j} \right\|$ and thus
\[
\|\left( \dot{e}_{I} \right)_{i, j}\| \leq \sum_{j' = 0}^{j}\left\| \left(\dot{e}_I^{\Delta}\right)_{i, j'} \right\| \leq^C \sum_{j' = 0}^j \cosh^{\mu} (j') \leq^C \cosh(j).
\]
This then gives the estimate $\left\| \dot{\Lambda} (\dot{e}_I) + e_{*,I} \right\| \leq^C \cosh^{\mu} (j)$.  We have thus reduced to the case of showing that $\dot{\Lambda}$ is a bounded isomorphism  in the norm $\| -\|_{\mu}$ from $\ul{e}_{II}$ into $\ul{e}_{*,II}$, which is Corollary \ref{NoKernelISo}. 
\end{proof}
As a direct corollary to the apriori estimate recorded in Proposition \ref{SimAprioriEstimate} and the existence result in Lemma  \ref{GrowingIso}, we have:

\begin{proposition}\label{VerticalVariations}
The mapping $\dot{\Lambda}: \ul{e} \rightarrow \ul{e}_{*}$ is a bounded  isomorphism in the weighted norms $\| -\|^{\sim}_{-\mu}$.
\end{proposition}
\begin{proof}
Fix $e_* \in \ul{e}_{*}$ with $\| e\|^\sim_{-\mu} = 1$. By Lemma \ref{Prop:WeightedNorm}, we have $\| e_*\| \leq^C \| e \|^{\sim}_{-\mu} \leq C$ and thus in particular $\| e_*\|_{\mu} \leq C$. By Lemma \ref{GrowingIso}, there is $e \in \ul{e}$ with $\| e\|_{\mu} \leq {C'}$ and such that $\dot{\Lambda} (e) = e_*$. By Proposition \ref{SimAprioriEstimate}, we have the estimate
\[
\left\| e\right\|^{\sim}_{-\mu} \leq^C \|e_*\|^{\sim}_{-\mu}.
\]
This completes the proof. 
 \end{proof}
 
 Thus, we have shown that the  operator is bounded in $C^1$ on a ball about the origin in $\ul{\xi}$ as a map into $\ul{e}_*$ and that the linearization at the origin is an isomorphism of $\ul{e}$ onto $\ul{e}_{*}$. As a direct application of the implicit function theorem we then have that the zeroes of $\Lambda$ near the origin are differentiably parametrized by the free variables $v \in \ul{v}$.

\begin{corollary} \label{SymmetryBroken}
There is $\epsilon > 0$ and $C > 0$ such that: Given $v \in \ul{v}$ with $\| v \|^\sim_{\mu} \leq \epsilon$, there is a unique $e \in \ul{e}$ with $\| e\|^{\sim}_{-\mu} \leq^C \epsilon $ and  such that 
\[
\Lambda_{v + e} = 0.
\]
Moreover, the mapping  $ v \mapsto  e$ is differentiable from $\ul{v}$ into $\ul{e}$. 
\end{corollary}
\begin{proof}
This is a direct consequence of Propositions \ref{OperatorinWeightedNorm},  \ref{VerticalVariations} and the implicit function theorem.
\end{proof}

\begin{remark} 
For the remainder of the article, we will regard $e \in \ul{e}$ as a differentiable function of $v \in \ul{v}$  using Corollary \ref{SymmetryBroken}, so that $\Lambda_{v} = \Lambda(v + e) = 0$ for each small value of $v$. Thus, each branch in $\mc{C}_{v}$ is a smooth surface with boundary comprising the waists of the incident half necks. 

\end{remark}

\section{Cancelling the translational error}\label{CanTranErr}

 In this section we resolve  the translational differences between the half necks. The result is a smooth immersed $CMC$ surface without boundary.

 \subsection{Regular parameter assigments}\label{RegParAss} It will be convenient to paramatrize the horizontal necks in $\mc{C}$ by their separation  $\sigma^{h}_{i, j}$ parameters, rather than the flux $\tau^{h}_{i, j}$, since they are regular near zero. Thus, a \emph{regular parameter assignment} in $\mc{N}^{h}$ is a tuple $r = (r_{i, j})$ where $r_{i, j} = \left(\sigma^{h}_{i, j}, \phi^{h}_{i, j} \right)$ and where $\sigma^h_{i, j}$ denotes the separation of the horizontal neck $\mc{N}^h_{i, j}$ and $\phi^{h}_{i, j}$ the phase. We denote by $\ul{r}$ the space of regular parameter assignments in $\mc{N}^h$. Observe that since the separation and flux parameters are related by (\ref{NewSigmaFromTau}), we have that $\ul{r}$ and the space $\ul{v}$  are isomorphic under the mapping 
 \begin{align}\label{TheMappingL}
 l: \left(\tau^{h}_{i, j}, \phi^{h}_{i, j} \right) \in \ul{v} \mapsto \left( l_{\tau^{h}_{i, j}}, \phi^{h}_{i, j}\right) \in \ul{r}
 \end{align}

\subsection{The translational error operator $\mc{E}$}
We let $\ul{r}_*$ denote the space of tuples in $\mb{R}^2$ indexed by pairs of integers $(i, j) \in \mb{Z} \times \mb{Z}$. We define an operator $\mc{E}: \ul{r} \rightarrow \ul{r}_*$ as follows: Given $(i, j) \in \mb{Z} \times \mb{Z}$, we let $\mc{N}^{h, +}_{i, j}$ and $\mc{N}^{h, -}_{i, j}$ denote the half necks in $\mc{N}^{h}_{i, j}$ that are incident to $\mc{S}_{i, j + 1}$ and $\mc{S}_{i, j}$, respectively, and we let $w^{h,  +}_{i, j}$ and $w^{h, -}_{i, j}$ denote the waists of each. Then $w^{h,  +}_{i, j}$ and $w^{h, -}_{i, j}$ differ by a translation, which we denote by $\mc{E}_{i, j} \in \mb{R}^2$. Thus we have 
\[
w^{h, +}_{i, j} = w^{h, -}_{i, j} + \mc{E}_{i, j}.
\]
This defines a  mapping $\mc{E}: \ul{r} \rightarrow \ul{r}_{*}$, which is locally differentiable in the sense that for each $(i, j)$  the mapping $\mc{E}_{i, j} $ is differentiable from $\ul{r}$ into $\mb{R}^2$.

\subsection{The restriction of $\mc{E}$}
 We restrict the domain and codomains $\ul{r}$ and $\ul{r}_*$ as follows:  Fix a finite subset $\mb{F}$ of $\mb{Z} \times \mb{Z}$.  We  will assume throughout that $r \in \ul{r}$  satisfies $r_{i,j } = 0$ unless $(i, j) \in \mb{F}$, and simillarly for $r_* \in \ul{r}_*$.  Let  $\mb{F}_i \subset \mb{Z}$ be defined by the condition that $j \in \mb{F}_i$ if and only if $(i, j) \in \mb{F}$. We will assume throughout that $\mb{F}_i$ is empty for $i \leq 0$ and $i >n $ and is otherwise nonempty.  
 Since the parameter $r_{i, j}$ vanishes for $i$ outside of the range $\{1, \ldots n \}$, so too does $v_{i, j} = l^{-1}_{r_{i, j}}$  and thus the branches $\mc{B}_{i}$, $i \in \mb{Z} \setminus \{ 0, \ldots,  n\}$ are smooth Delaunay surfaces.
 \begin{remark} \label{CRestriction}
  For the remainder of the article, we will restrict $\mc{C} = \mc{C}_{v}$ to the sub-collection comprising the branches $\mc{B}_{0}$ through $\mc{B}_{n}$.
\end{remark}
\subsection{Regular variables generate the error space $\ul{r}_{*}$}
 When $r = 0$, the branches of $\mc{C}_{r}$ are rotationally symmetric Delaunay ends, differing from each other by a translation. We can assume that a fixed translation has been applied separately to the branches so that $\mc{E}$ vanishes at $r = 0$. Thus, the collection $\mc{C}_r$ at $r = 0$ is a collection of Delaunay ends touching tangentially.

We wish to will apply the implicit function theorem, with $r$ serving as the regular variable, and with families of translations of the branches serving as the free parameters. We show now that the linearization of $\mc{E}$ at $r = 0$ in $r$ is an isomorphism onto the error space $\ul{r}_*$. In order to prove embeddedness of some the surfaces we construct, we will need to precisely understand the effect of the parameters $\phi^h$ and $\sigma^h$. In the following, we let $\dot{\mc{E}}$ denote the linearization of $\mc{E}$ at $r = 0$.

\begin{proposition}\label{EIsAnIso}
The mapping $ \dot{\mc{E}}: \ul{r} \rightarrow \ul{r}_{*}$ is an isomorphism.  Precisely, it holds that 
\[
 \left\{\dot{\mc{E}} \left(\dot{\sigma}^h\right) \right\}_{i, j}= - \dot{\sigma}^h_{i, j} e_y, \quad \left\{\dot{\mc{E}}(\dot{\phi}^{h})\right\}_{i, j} = \dot{\phi}_{i, j}^{h} e_z + O(\tau_0).
\]
\end{proposition}

\begin{proof}
In the following, we will fix a pair $(i, j)$ in $\mb{F}$ and suppress it from the notation.  
We assume that the coordinate axes are positioned so that the waists $w^- : = w^{h, -}$ and $w^+ :   =w^{h, +}$, which are points when $ v = 0$,  coincide with the origin. Recall that  throughout we have fixed a base flux assignment $\tau_0 \in \ul{\tau}_0$ and suppressed $\xi_0 : = \tau_0 + e_{\mc{G}, \tau_0}$ from our notation. Throughout, we have then written $\mc{C}_{\xi}$ to mean $\mc{C}_{\xi + \xi_0 }$. Thus, making the dependence on all parameters explicit again, we have $\mc{C}_{r} = \mc{C}\left(v_r + e_{v_r} + \xi_0\right)$, where we have put $v_{r} = l^{-1}_{r}$ and where $l$ is the mapping defined in Section \ref{RegParAss}. Thus, when  $\tau_0 = r= 0$, we have that the total parameter  $\xi_{\text{total}}$ vanishes and the branches of $\mc{C}$ coincide with  the collection $\mb{S}$ of spheres of radius $1/2$ on the integer lattice $(0, i, j)$. The mapping $v \mapsto e_v$ is not  defined here--it is only defined for positive base flux $\tau_0$--but we can compute the partial derivative of $\mc{E}$ in $\phi^h$ by fixing the remaining parameter values at $0$. The necks $w^{-}$ and $w^+$ are then given by
\[
w^- = \frac{1}{2}\left(\cos(\phi^h) e_y + \sin(\phi^h) e_z\right) + \left(0, -\frac{1}{2}, 0\right)
\]
and 
\[
w^+ = - \frac{1}{2}\left(\cos(\phi^h) e_y + \sin(\phi^h) e_z\right) + \left(0, \frac{1}{2}, 0\right)
\]
We then have 
\[
\mc{E}_{i, j} =  w^+ - w^- = \left(\cos(\phi^h) e_y + \sin(\phi^h) e_z\right) + \left(0, 1, 0\right),
\]
From which the partial derivative $\frac{\partial \mc{E}}{\partial \phi^h}$ can be computed. 

Now, we compute the total derivative $\mc{E}_{, \phi^h}$  for non-vanishing $\tau_0$ and vanishing $v$, where we are again regarding $e = e_v$ as a function of $v$. Since $\mc{E} = \mc{E}_{\xi}$ is a differentiable function of $\xi = v + e$ and $e = e_v$ is a function  of $v$, the total derivative of $\mc{E}$ in  $\phi^h$ is given by
\[
\mc{E}_{, \phi^h}  = \frac{\partial\mc{E}}{\partial \phi^h} + \frac{\partial \mc{E}}{\partial e } \frac{\partial  e}{\partial \phi^h}.
\]
Picking a variation $\dot{\phi}^h$ of $\phi^h$ at $\phi^h = 0$ and letting $\dot{e} = \left. D e \right|_{v = 0} (\dot{\phi}^h)$ denote the induced variation of $e$, we have. 
\[
\dot{\Lambda}\left(\dot{\phi}^h + \dot{e}\right)  = 0.
\]
Observe that $\left. \frac{\partial \Lambda}{\partial \phi} \right|_{\xi_{total} = 0} = 0$.  By the  differentiability properties of the mapping $\Lambda$ we then have  the estimate  $\left\|\dot{\Lambda}\left(\dot{\phi}^h\right)\right\| \leq^C \tau_0$.  Proposition \ref{SimAprioriEstimate} then gives  $\left\|\dot{e}\right\|^{\sim}_{-\mu} \leq^C \left\| \dot{\Lambda}\left(\dot{\phi}^h\right)\right\|^{\sim}_{-\mu} \leq^C \tau_0 \left\|\dot{\phi}^h\right\|$. Thus, we have
\[
\left\|\left. \mc{E}_{, \phi^h} \right|_{\tau_0} - \left. \mc{E}_{, \phi^h} \right|_{\tau_0 = 0} \right\| \leq^C \tau_0.
\]
For the partial derivative $\mc{E}_{, \sigma^h}$, we can similarly compute the total derivative of $\mc{E}$ in $\sigma^h$ as
\begin{align} \label{TotalDD}
\mc{E}_{, \sigma^h} = \frac{\partial \mc{E}}{\partial \sigma^h} + \frac{\partial \mc{E}}{\partial e} \frac{\partial e}{\partial \sigma^h}
\end{align}
Directly from the definition of the separation parameter $\sigma^h$, we have $\frac{\partial \mc{E}}{\partial \sigma^h} = -\dot{\sigma}^h e_y$. Recall that the flux $\tau$ and separation $\sigma$ parameters   are related to each other by (\ref{NewSigmaFromTau}). Thus, at $\tau= \sigma = 0$ we have $\frac{d \tau}{d \sigma} = 0$. Applying this to  (\ref{TotalDD}) we have that that at $v = 0$ we have $\mc{E}_{, \sigma^h} = \frac{\partial \mc{E}}{\partial \sigma^h} =- \dot{\sigma}^h e_y$. This completes the proof. 
\end{proof}

\subsection{Immersed surfaces parametrized by Branchwise translations}
As a direct consequence, we can construct families of smoothly immersed CMC surfaces $\mc{C}_{\mc{M}}$ parametrized by branchwise motions of  $\mc{C}$. Here, a \emph{branchwise} motion is a mapping $\mc{M}$ on $\mc{C}$ such that the restriction to each branch is a rigid motion. In the following, if $\mc{M}$ is a branchwise translation of $\mc{C}$,  we let $\mc{C}^{\mc{M}}$ denote the image of $\mc{C}$ under $\mc{M}$. We also let $\mc{E}^{\mc{M}}_{v}$ denote the translational error between the waist pairs in $\mc{C}^{\mc{M}}$.  It will be convenient to parametrize the family of branchwise translations by their \emph{relative} translations, since applying a single translation to the collection $\mc{C}$ does not effect $\mc{E}$. We do this as follows: Given constants $d = (d_i)$ with $d_i > 0 $, $i > 1$, we let $\mc{M} = \mc{M}_d$ denote the unique branchwise translation in $\mc{C}$ such that $\left. \mc{M} \right|_{\mc{B}_i} =d^i e_y $ and $\left. \mc{M} \right|_{\mc{B}_0}$ is the identity. Here we have set $d^i : =  \left(\sum_{j \leq i} d_j \right) $. We will refer to $d$ as the \emph{relative translation} parameter, since the motion $\mc{M}^{d}$ determined by $d$ specifies the relative translations between branches, rather than the absolute translation. We then regard $d$ as the free parameter of  $\mc{C}$ by setting $\mc{C}_{d} = \mc{C}^{\mc{M}^{d}}$ and we put  $\mc{E}^d_{r} := \mc{E}^{\mc{M}^d}_{r}$

\begin{corollary}\label{Branches!}
There is a differentiable mapping $d \mapsto r_{d} \in \ul{r}$ such that $\mc{E}^{d}_{r_{d}} = 0$.
\end{corollary}
\begin{proof}
This is a direct consequence of Proposition \ref{EIsAnIso} and the implicit function theorem.
\end{proof}

 \section{Embeddedness of solutions}\label{EmbeddSol}
  The surfaces   $\mc{C}_{d}$ are  smoothly immersed CMC surfaces without boundary depending differentiably on $d$. In this section we record conditions on the relative translation parameter $d$ such that the surface $\mc{C}_d$ is embedded. Embeddedness on compact subsets of $\mb{R}^3$ is fairly easy to establish and follows from basic continuity arguments.  The asymptotic embeddedness of the surfaces is established by a careful analysis of the linearization of the parameters $v$ and $e$ in $d$.

 \subsection{Notational conventions in this section}
 Throughout this section, we will regard $d = (d_i)$ as a fixed choice of relative translations, which we  assume small relative to various considerations. We  will finally  exhibit an explicit choice for $d$ in the proof of Proposition \ref{EmbeddedChoice} recorded at the end of the section. We  will set $\tilde{\sigma}_{i, j} = d_i$ and we let $\tilde{\tau}_{i, j}$ be defined by the relation  (\ref{NewSigmaFromTau}).  We put $\tilde{r} = \left( \tilde{\sigma}^h_{i, j}, 0 \right)$ and $\tilde{v} = \left(\tilde{\tau}^h_{i, j}, 0\right)$. Since $\tilde{\sigma}_{i, j}$, and consequently $\tilde{v}_{i,j}$ depend only on $i$, we will in places write $\tilde{\sigma}^h_{i}$ instead of $\tilde{\sigma}^h_{i, j}$ and similarly for $\tilde{\tau}^h_{i, j}$.   We let $ r = r_{d} = \left(\sigma^h_{i, j}, \phi^h_{i, j} \right)$  be determined by  $d$ using Corollary \ref{Branches!}  above, and we  set $v = l^{-1}_{r} = (\tau^{h}_{i,j}, \phi^{h}_{i, j})$. 

\subsection{Existence of embedded surfaces}
The main theorem of this section is:
\begin{proposition} \label{EmbeddedChoice}
There exist relative translations $d_i$ such that the  surface $\mc{C}_{d}$ is embedded. 

\end{proposition}

We establish embeddedness in two steps. The first step, which is relatively simple, establishes embeddedness on compact subsets of $\mb{R}^3$.  The only conditions required in this step on the relative separations $d_i$ are that they are positive, and embeddedness follows from the fact that $\tilde{\tau}_{i, j}$ is much smaller than $d$. Since the solution $e = e_{v}$ depends differentiably on $v$, which is much smaller than $d$, the resulting correction is small relative to $d$ and embeddedness is thus preserved locally.   For asymptotic embeddedness, additional conditions are required. To understand these, we need to understand carefully the asymptotics of the ends of $\mc{C}_{d}$ for a fixed choice of $d$. 

To do this, we extract asymptotic values of the solution $e = e_{v}$.  By Corollary \ref{SymmetryBroken}, we have that $\| e\|^{\sim}_{- \mu} \leq^C \left\| v\right\|$. Thus, for fixed $i$,  we have $e^{\Delta}_{i, j} \cong \left\|v \right\|\cosh^{- \mu} (j)$, and thus $e_{i, j}$ is a cauchy sequence in $j$. We let $\ol{e}_{i} = \lim_{j \rightarrow \infty} e_{i, j}$ denote its limit, which we regard as periodic in $j$ by setting $\ol{e}_{i,j} = \ol{e}_{i}$. By continuity, we have $\Lambda (\ol{e}) = 0$, and thus the corresponding limit $\ol{\mc{B}}_{i}$ of the branch $\mc{B}_{i}$ is a smooth  periodic CMC surface in $\mb{R}^3$. Such a surface is necessarily a Delaunay end, which we prove with tools internal to this paper. Once it is established that $\ol{\mc{B}}_i$ is symmetric about an axis, the asymptotics  of $\ol{\mc{B}}_i$ can be determined precisely from the limit data $\ol{e} = \left(\ol{\tau}, \ol{\phi}, \ol{\rho}, \ol{f} \right)$. Precisely, it follows that the axis of $\ol{\mc{B}}_{i}$ is parallel to $ -  \sin (\ol{\phi}_i)e_y + \cos (\ol{\phi}_i) e_z$.  It then follows that the surface $\mc{C}_d$ is asymptotically embedded if $\ol{\phi}^{\vtr} <0$. We first record, in Lemma \ref{GoodApproximation} below, bounds for the exact solutions $r$ in terms of the approximate solution $\tilde{r}$.

\begin{lemma} \label{GoodApproximation}
It holds that $\| r- \tilde{r}\| \leq^C \| \tilde{\tau}\|$.

 \end{lemma}
 \begin{proof}

In the following, we will make the dependence of  $\mc{E}$ on $e$ explicit again, so that  $\mc{E}^{d}_{r} = \mc{E}(r + e_r)$. Observe that  $\mc{E}^{d}(\tilde{r}) = 0$.  Too see this, observe that, when $r$ and $e$ vanish, the branches $\mc{B}_i$ and $\mc{B}_{i + 1}$ of $\mc{C}$ differ from each other by the translation by $e_y$  and are symmetric through their axes, respectively. Thus, the horizontal necks $\mc{N}^{h}_{i, j}$ are symmetric through vertical lines passing through their centers. When the parameter values $e$ and $\phi^h$ are zero, the variations induced by $\tau^{h}$ are supported on the horizontal catenoids and preserve the symmetries listed above. Thus, the boundary components  $\partial^{\pm} \mc{N}^h_{i, j}$ of  $ \mc{N}^h_{i, j}$ differ by translations  $\mc{t}^{\pm}$  given by $ \mc{t}^{\pm} = \pm\sigma^{h}_{i, j} e_y$. Since $\tilde{\sigma}^{h}_{i,j} = d_i$ we can apply the  branchwise translation $\mc{M}_{d}$ determined by $d$ to $\mc{C}_{\tilde{r}}$ so that  $\mc{E}^{d}_{\tilde{r}} = 0$ as claimed.

Now, since $e = e_v$ is a differentiable function of $v$ we have:
\[
\mc{E}_{\tilde{r}} =  \mc{E}^{d}(\tilde{r} + e_{\tilde{v}}) = O(e_{\tilde{v}}) = O \left(\left\|\tilde{v}\right\|\right) = O \left( \|\tilde{\tau}\|\right).
\]
Since $\mc{E}^{d}_{r} = 0$ we have
  \begin{align*}
0&  = \mc{E}^{d}_{r} - \mc{E}^{d}_{\tilde{r}} + O \left(\|\tilde{\tau}\|\right) \\
& = \left(\int_{0}^1\left. \frac{\partial \mc{E}}{\partial r} \right|_{r(t)} dt \right)\left( r - \tilde{r}\right)  + O (\|\tilde{\tau}\|),
\end{align*}
where above we have set $r(t) = (1 - t) \tilde{r} + t r$.  Taking $d$ small, we can arrange for $r$ and $\tilde{r}$ to arbitrarily close to $0$, so that the operator  $\int_{0}^1\left. \frac{\partial \mc{E}}{\partial r} \right|_{r(t)} dt $ is arbitrarily close to $ \left. \frac{\partial\mc{E}}{\partial r} \right|_{0}$. Applying $\left(  \int_{0}^1\left. \frac{\partial\mc{E}}{\partial r} \right|_{v(t)} dt\right)^{-1}$ to the right hand side above then gives that $r - \tilde{r} = O\left(\left\|\tilde{\tau}\right\| \right)$. This completes the proof.
\end{proof}
As a direct consequence of Lemma \ref{GoodApproximation}, we have the estimate 
\begin{align} \label{VPrimeBound}
\left\| v \right\| \leq^C \left\|\tilde{\tau}\right\|.
\end{align}

In order to determine precisely the asymptotics of the ends in terms of their limits, we need to know that the ends are rotationally invariant about an axis, and not just that they are singly periodic.  This is recorded in Lemma \ref{RotSymEnds} below:

\begin{lemma}\label{RotSymEnds}
Each asymptotic end $\ol{\mc{B}}_i$ is invariant under rotations about an axis.
\end{lemma}
\begin{proof}
We fix $i$ and suppress it from our notation throughout this proof, and thus we will write $\ol{\mc{B}}$ for $\ol{\mc{B}}_i$.  Observe that, since the limit $\ol{\mc{B}}$ is singly periodic in the direction $ \ol{\mc{t}} : = - \sin(\ol{\phi}) e_y + \cos (\ol{\phi}) e_z$, the axis is determined uniquely up to translations. In fact, we can uniquely define an axis in various ways, such as, for example, declaring it to be the unique line parallel to the direction $v$ minimizing the $L^2$ distance over a fundamental domain. For the purposes of the proof, this is not necessary and we take, for the time being, the axis to be the line  parallel to $\ol{\mc{t}}$ and passing through the origin. 

 Now, suppose the claim is false. Then the normal variation field generated by rotations about the axis has a non-zero part  $u$ orthogonal to the jacobi fields generated by translations over fundamental domains of $\ol{\mc{B}}$. The function $u$ is then a periodic jacobi field on $\ol{\mc{B}}$. We can assume that $u$ normalized  so that the supremum is $1$. Taking $d \rightarrow 0$, the asymptotic solution $\ol{e}$ tends to zero and thus the asymptotic end $\ol{\mc{B}}$ converges on compact subsets of $\mb{R}^3$ to a  smooth CMC surface of rotation with axis parallel to $e_z$ (Thus, it is a Delaunay surface). After possibly applying a translation, we can assume that the axis of rotation coincides the $z$-axis. The corresponding solutions   converge to a non-trivial periodic jacobi field  $u$ on a Delaunay surface with small parameter value $\tau_0$, which is orthogonal to the translational jacobi fields over fundamental domains. The proof is finished by arguing as in the proof of Proposition \ref{RotSymBranches}.  Though the overall strategy is the same, there are enough differences that we present the argument completely. We normalize $u$ so that $\sup u = 1$.   We can decompose $u$ into its higher and lower parts $\mr{u}$ and $\ol{u}$. The lower part lies in the six dimensional space of geometric jacobi fields generated by  rotations, translations, and the variation of the Delaunay parameter. Since the Delaunay variation and rotations generate linearly growing jacobi fields, it follows that $\ol{u}$ is generated by a translation. However, since $u$ is  orthogonal to the translational jacobi fields over fundamental domains of  $\ol{\mc{B}}$, we have that  $\ol{u}$ must vanish and thus $u = \mr{u}$. Now, take a sequence $\tau_{0, j}$ of base flux assignments tending to  $0$ and consider the corresponding sequences  $\ol{\mc{B}}_j$ of branches and  $u_{j}$ of periodic jacobi fields  on $\ul{\mc{B}}_j$. Observe that as $\tau_{0, j} \rightarrow 0$ the branch $\ol{\mc{B}}_j$ converges to a  limit $\ol{\mc{B}}_{0}$, which is a collection of spheres of radius $1/2$ touching at the north and south poles--that is, the points where the unit normal coincides with $\pm e_z$.  Pick a sequence $p_j$ of points realizing the supremum of $u_j$, so $1 = u_j (p_j)$. There are then several possibilities. If $p_j$ remains in a fixed compact subset away from the poles, then $u_j$ converges to a non-trivial  bounded jacobi field on the sphere away from the north and south poles. Standard removable singularity theory then implies that it extends smoothy across the puncture to a non-trivial jacobi field on the sphere, and is thus a  translational jacobi fields. However, this is a contradiction, since we assumed that $u$ is orthogonal to the translational jacobi fields over each fundamental domain. Thus, we must have that $p_j$ converges to a pole, which without loss of generality we can take to be the origin. After applying a small translation in the $e_z$ direction, we can assume that $p_j$ lies in the plane $\{z = 0 \}$. The rescaled Delaunay surface $\tilde{\mc{B}}_i = \ol{\mc{B}}_i/ |p_i|$ then converges to a limit $\tilde{\mc{B}}$ that is either a flat plane or else the standard catenoid. In both cases, define the function $\tilde{u}_j$ on $\tilde{\mc{B}}_j$ by $\tilde{u}_j (p) = u(|p_j| p)$. Then $\tilde{u}_j$ is a jacobi field on  $\tilde{\mc{B}}_j$ that achieves its supremum, which is equal to $1$, on the unit circle. The sequence is then uniformly bounded in $C^k$ and converges to a nontrivial   bounded jacobi field  $u$ on $\wt{\mc{B}}$. In the case that $\wt{\mc{B}}$ is the punctured plane, $u$ extends smoothly across the origin  to a bounded harmonic function  on the plane, and is thus constant. However, since $u$ is the limit of functions with vanishing lower part, the lower part of $u$ vanishes as well, which gives a contradiction. On the other hand, if the limit $\wt{\mc{B}}$ is the standard catenoid, then we obtain a contradiction to Lemma \ref{CatJacFields}. This completes the proof.

 \end{proof}  

As a consequence of Lemma \ref{RotSymEnds}, the asymptotics of the ends can be determined explicitly from the asymptotic data $\ol{e}$. In particular, if $\ol{\phi^v}$ is the asymptotic phase of the vertical catenoids, then the axis of rotation of $\ol{\mc{B}}_i$ is parallel to $\ol{\mc{t}}_i$ (defined in the proof of Lemma \ref{RotSymEnds} above). As a consequence, the problem of proving asymptotic embeddedness of the surfaces $\mc{C}_d$ is reduced to the problem  of prescribing the asymptotic vertical phase. 

\begin{lemma} \label{EmbeddedCriteria}
If $\ol{\phi^{v, \vtr}}  < 0$ and $|d|$ is sufficiently small, then $\mc{C}_{d}$ is embedded. 
\end{lemma}
\begin{proof}
Since the embeddedness of $\mc{C}_{d}$ on compact subsets  follows easily from continuity arguments, we discuss here only the asymptotic embeddedness of the surfaces $\mc{C}_{d}$.
Let $\mc{M}_{i, j}$ denote the piecewise translation of $\mc{C} = \mc{C}_{d}$ that resolves the translational errors along the branches. We wish to compare $\mc{M}_{i, j}$ with the asymptotic translations $\ol{\mc{M}}_{i, j}$ resolving the translational errors the asymptotic branches  $\ol{\mc{B}}_i$. We will use the notation  $\mc{M}^{\Delta}_{i, j}$ to denote the difference between the translations on leaves $\mc{L}_{i, j}$ and $\mc{L}_{i, j - 1}$.  It follows from the periodicity of the asymptotic end $\ol{\mc{B}}_i$ that   $\ol{\mc{M}}^{\Delta}_{i, j} = \ol{v}_i$ is a fixed vector independent of $j$ that depends continuously on  the axis and the Delaunay parameter of  $\ol{\mc{B}}$. When the Delaunay parameter vanishes, a straightforward computation gives that $\ol{\mc{M}}^{\Delta}_{i, j} = \ol{v}_i = \left(0, - \sin(\ol{\phi}_i), \cos(\ol{\phi}_i) - 1 \right)$.   Moreover, since we have the estimate $\| e_{i,j} - \ol{e}_{i, j}\| \leq^C \left|\tilde{\tau}\right| \cosh^{- \mu} (j)$,  it follows that $\left\| \mc{M}^{\Delta}_{i, j} - \ol{\mc{M}}^{\Delta}_{i, j}\right\| \leq^C |\tilde{\tau}| \cosh^{- \mu} (j)$. Thus we have 
\begin{align*}
\mc{M}_{i, j} & =  \mc{M}_{i, j} - \ol{\mc{M}}_{i, j} \\
&  = \sum_{j' = 0}^{j} \mc{M}^{\Delta}_{i, j'} - \sum_{j'  = 0}^{j}\ol{\mc{M}}^{\Delta}_{i, j'}\\
& = j \ol{v}_i + O (|\tilde{\tau}|).
\end{align*}
It follows from the above estimates for the motions $\mc{M}_{i, j}$ that away from a compact set about the origin, the branch $\mc{B}_i$ is the contained in the region  $R_i$ defined by
\[
- \tan(\ol{\phi}_{i})|z| + d^i +  i - \frac{1}{2}- O (\tilde{\tau}) \leq y \leq - \tan(\ol{\phi}_{i})|z| + d^i +  i  + \frac{1}{2}- O (\tilde{\tau}).
\]
For $d$ sufficiently small we have that $\left|\tilde{\tau}\right| < < \left|d\right|$. Since $\ol{\phi}_{i + 1} < \ol{\phi}_i$, and for $d$ sufficiently small, the regions $R_i$ and $R_{i + 1}$ do not intersect: Given $(x, y, z)$ in $R_i$, the inequality above must hold. However, we have
\begin{align*}
- \tan(\ol{\phi}_{i})|z| + \frac{1}{2} + i + d^i + O (\tilde{\tau}) &  <- \tan(\ol{\phi}_{i + 1}) |z|  - \frac{1}{2} + (i + 1) + d^{i + 1} + O (\tilde{\tau})
\end{align*}
and thus $(x, y, z)$ cannot lie in both regions.  This completes the proof. 
\end{proof}
We are now ready to prove Proposition \ref{EmbeddedChoice}.

\begin{proof}[Proof of Proposition \ref{EmbeddedChoice}]
We are regarding $e \in \ul{e}$ as a function of $v$ so that $\Lambda_{v + e} = 0$. Since $\Lambda$ is differentiable, we have $\dot{\Lambda}_{v+ e} = o \| v + e\| = o (\| v\|)$ and thus
\begin{align*}
e & = - \dot{\Lambda}^{-1, e} \left( \dot{\Lambda} (v)\right) + o (\| v\|) \\
& = : \dot{e} + o (\| v\|) \\
& = \dot{e} + o(\left\|\tilde{\tau}\right\|)
\end{align*}
where $\dot{e} = (\dot{\tau}^v, \dot{\phi}^v)$ is defined implicitly above and where the term $o(\| \tilde{\tau} \|)$ is measured in the norm $\| -\|^\sim_{\mu}$.
Since $\mb{F}$ is a finite set, there is and $j_0$ such that $(i, j) \in \mb{F}$ implies  $j < j_0$. From Lemmas \ref{TauProjections} and \ref{Projections} we have $\dot{\phi}^{v,\Delta}_{i, j} \cong \dot{\tau}^{h, \vtr}_{i, j }$, so that for $j > j_0$ we have
\[
\dot{\phi}^v_{i, j}  \cong   \mc{b}_{i},
\]
where above we have set $\mc{b}_i = \sum_{j = 0}^{j_0}\tau^{h, \vtr}_{i, j }$. In particular, if 
\begin{align} \label{BDifference}
\mc{b}_{i + 1} - \mc{b}_{i} \leq - c \| \tilde{\tau} \|,
\end{align}
for a positive constant $c$, we have $ \dot{\phi}^{h, \vtr}_{i , j } \leq - c \| \tilde{\tau}\| + o (\| \tilde{\tau}\|)$ for $j > j_0$. The weighted estimates then imply that $ \ol{\phi}_i^{h, \vtr} < 0$ and thus  $\mc{C}$ is asymptotically embedded for $d$ sufficiently small.

From Lemma \ref{GoodApproximation} we have.
\[
\sigma^h_{i, j}   = d_i + O\left( \left\|\tilde{\tau}\right\| \right), 
\]
which then directly gives:
\[
\tau^h_{i, j} = \tilde{\tau}^h_{i, j} + \int_{0}^{1} \tau'_{d (t)}\left(O (\left\|\tilde{\tau}\right\|) \right) dt, \quad d(t) = d_i + t O (\tilde{\tau}).
\]
Dividing by $\tilde{\tau}^h_i$, taking $d \rightarrow 0$ and using that the derivative $\tau'_{d}=$  of $\tau$ as a function of $d$ tends to zero as $d \rightarrow 0$ gives
\[
 \rho : = \left|\frac{\tau^h_{i, j}}{\tilde{\tau}_{i}} - 1 \right| \rightarrow 0.
\]
For a constant $c  >0$ sufficiently small, define numbers $d_i$ by the condition
\begin{align}\label{ChoiceOfRelativeTranslations}
 \frac{i (i - 1)}{2} c : = \left|\mb{F}_i\right| \tilde{\tau}_{i},
\end{align}
 Observe that $\| \tilde{\tau}\| \cong c$. We have
\begin{align*}
\mc{b}_i & = \sum_{j = 0}^{j_0} \tau^{h, \vtr}_{i, j} \\
& =  \sum_{j = 0}^{j_0} \tau^{h}_{i + 1, j } - \sum_{j = 0}^{j_0} \tau^{h}_{i, j} \\
& = \left(\sum_{j = 0}^{j_0} \frac{\tau^{h}_{i + 1, j }}{\tilde{\tau}_{i + 1}} \right) \tilde{\tau}_{i + 1}-  \left(\sum_{j = 0}^{j_0} \frac{\tau^{h}_{i + 1, j }}{\tilde{\tau}_{i}} \right) \tilde{\tau}_{i} \\
& =  c \left(\frac{i (i + 1)}{2}  - \frac{i (i - 1)}{2}\right)+  E \\
&  =  ci + E,
\end{align*}
where the error term $E$ above is explicitly given by
\[
E : = \sum_{j = 0}^{j_0} \left( \frac{\tau^{h}_{i + 1, j }}{\tilde{\tau}_{i + 1}}  - 1\right) \tilde{\tau}_{i + 1} -  \sum_{j = 0}^{j_0} \left( \frac{\tau^{h}_{i + 1, j }}{\tilde{\tau}_{i}}  - 1\right) \tilde{\tau}_{i}.
\]
In particular $E$ satisfies the estimate
\[
|E| \leq \rho c \left( \frac{i (i - 1)}{2} + \frac{i (i + 1)}{2}\right)(2 j_0 + 1).
\]
Observe that $\rho \rightarrow 0$ as $c \rightarrow 0$. Thus for $c$ sufficiently small we have   $| E| \leq \frac{c}{4}$ for all $i < i_0$.  This then gives
\begin{align*}
\mc{b}_{i + 1} - \mc{b}_i  &  \leq \frac{-c}{2} \\
& \leq \tilde{c} \| \tilde{\tau}\|
\end{align*}
where the last equality follows from the fact that $\| \tilde{\tau}\| \cong c$. This gives (\ref{BDifference}) and  completes the proof.
\end{proof} 
\subsection{Adding one end to the construction} \label{AddEnd}

The surfaces $\mc{C}_d$ constructed in Corollary \ref{Branches!}  have   $2(n + 1)$ ends, where $n$ is an arbitrary postitive integer. In this section we explain a basic modification of the construction  that adds one end, yielding surfaces with $2n + 3$ ends. \

We let $\mb{S}'$ denote the collection of spheres $\mb{S}'_j  : = \mb{S}_{-j, 0}$ for $j < 0$, and we define a surface $\mc{C}'$ using $\mb{S}'$  and the strategy of Section \ref{PreEmSpheres}. Precisely, we take the singular points $\mc{s}'$ of the configuration to be the points of intersection $\mc{s}'_j : = \mb{S}'_j \cap \mb{S}'_{j + 1} = (0, - j -1/2, 0)$ along with the point $\mc{s}'_0 : = (0, - 1/2, 0)$, and we let $\mc{N}'$ and $\mc{S}'$ denote the singular and spherical parts of $\mb{S}'$, respectively. Observe that $\mc{N}'$ comprises the components $\mc{N}'_j$  containing the singular point $\mc{s}'_j$ for  $j < 0$ as well as the component $\mc{D}'_0$ containing $\mc{s}'_0$, and that $\mc{S}$ comprises the components $\mc{S}_j$ contained in $\mb{S}_j$. We let $\mc{C}'$ denote the  disjoint union of the components of $\mc{S}'$ and $\mc{N}'$.  Since we are assuming symmetry through the coordinate plane $\{z = 0 \}$,  we fix the phase assignments of the necks in $\mc{N}'_j$ at zero and the free parameters $\ul{\xi}'$ of  $\mc{C}'$ comprise the space $\ul{\tau}'$ of flux assignments of the necks in $\mc{C}'$ as well as the space $\ul{\rho}'$ of piecewise motions of $\partial \mc{S}'$ and the space $\ul{f}'$ of even functions in $\mr{C}^{2, \alpha} (\partial\mc{C}')$. For clarity, we remark that we are regarding $\mc{D}'_0$ as a disk, not half of a CMC neck, and thus we fix the flux assigment on $\mc{D}'_0$ at zero.  We let $*'$ denote the pairing on the boundary $\partial\mc{C}'$ of $\mc{C}'$ induced by the canonical projection into $\mb{R}^3$. The defect operator $\Lambda'$ associated to the collection $\left( \mc{C}', *'\right)$ is then a differentiable map  defined on a ball about  $0$  in $\ul{\xi}'$ into $\ul{e}_{*}'$, where the error space $\ul{e}'_*$ is the space of even $C^{1, \alpha}$ functions on $\partial \mc{C}'$, and where both spaces above are equipped with the norms $\| -\|^\sim_{- \mu}$.  Using the arguments of Section \ref{DefectOnDelaunay} and the symmetry assumptions, it follows that  
 
\begin{lemma}\label{InScaleZeroSymm}
$\slashed{\partial}^e$ is an isomorphism from $\ul{e}'$ onto $\ul{e}'_{*}$ in the norms $\| -\|^\sim_{- \mu}$
\end{lemma}
\begin{proof}
We repeat the proof of  Lemma \ref{GrowingIso} explaining the modifications needed as we proceed, which  are needed since the background parameter on $\mc{C}'$ is zero and thus Lemma \ref{Projections} does not apply. This is compensated by the fact that the symmetry assumptions a-priori yield orthogonality of the source terms to the translational jacobi field $\mc{k}^z$. The proof of Lemma \ref{InScaleZeroSymm} the proceeds identically to the proof of Proposition \ref{VerticalVariations}.
We first show that $\dot{\Lambda}_{\mc{C}'}$ is an isomorphism  in the norms $\| -\|_{\mu}$ for arbitrary $\mu > 0$. To see this, pick $e_* \in \ul{e}'_*$ with $\| e_*\|_{\mu} = 1$. By symmetry considerations, $e_{*, j}$ is orthogonal to the translational killing field $\mc{k}^{z '}_{j}: = \mc{k}^{z}_{-j, 0}$ for all $j$. Thus, by  Lemma  \ref{TauProjections} there is a unique $ \dot{\tau}' \in \ul{\tau}'$ such that  $\dot{\Lambda} ( \dot{\tau}') + e_{*} \in \ul{e}_{II}$. Moreover, we have the estimate  $\left\|\left(\dot{\tau}'^{\Delta} \right)_{ i, j} \right\| \cong \left\|e_{* i, j} \right\|$ and thus
\[
\|\left( \dot{\tau}' \right)_{i, j}\| \leq \sum_{j' = 0}^{j}\left\| \left(\dot{\tau}'^{\Delta}\right)_{i, j'} \right\| \leq^C \sum_{j' = 0}^j \cosh^{\mu} (j') \leq^C \cosh(j).
\]
This then gives the estimate $\left\| \dot{\Lambda} (\dot{\tau}' + e_{*}) \right\| \leq^C \cosh^{\mu} (j)$.  We have thus reduced to the case of showing that $\dot{\Lambda}$ is a bounded isomorphism  in the norm $\| -\|_{\mu}$ from $\ul{e}_{II}$ into $\ul{e}_{*,II}$, which is Corollary \ref{NoKernelISo}. The proof is then finished by appealing to the  Proposition \ref{VerticalVariations}.
\end{proof}

In the following, we will fix a choice $d_0$ of relative translations of $\mc{C}$ constructed in the Proof of Proposition \ref{EmbeddedChoice}.    In particular, we will assume that $d_0$ is given by (\ref{ChoiceOfRelativeTranslations}). Recall that $r_0 = r_{d_0}$ is determined by $d_0$ by Corollary \ref{Branches!} and $e_0 = e_{v_0}$ is  determined by $v_0$ by Corollary \ref{SymmetryBroken}. Thus,  $e_0$ is determined by  $d_0$ using   the composition 
\[
e_{d} = e_{v_d}, \quad v_d : = l^{-1}_{r_d}
\]
where above the mapping $l: \ul{v} \rightarrow \ul{r}$ is given by (\ref{TheMappingL}). We will in the following re-introduce $\xi \in \ul{\xi}$  as a  free variable of the construction  by setting
\[
\mc{C}_{\xi} : = \mc{C}_{\xi + e_0 + v_0}.
\]
Observe that we are still suppressing the background parameter $\xi_0$ from the notation, and thus the total parameter value is given by $\xi_{total} = \xi_0 + e_0 + v_0 + \xi$. With this notation, we have that $\left. \mc{C}_{\xi}\right|_{\xi = 0}$ is a complete embedded CMC surface and that the ends are asymptotically separated.

Now, consider the defect operator $\Lambda_{\mc{C} \coprod \mc{C}'}$ associated to the disjoint union of the collections $\mc{C}$ and $\mc{C}'$. We can regard $\Lambda_{\mc{C} \coprod \mc{C}'}$  as isomorphic to the direct sum   $\Lambda_{\mc{C}'} \oplus \Lambda_{\mc{C}}: \ul{\xi'} \oplus \ul{\xi} \rightarrow \ul{e'}_{*} \oplus \ul{e}_{*}$. By Proposition \ref{VerticalVariations}  and  Lemma \ref{InScaleZeroSymm}, and taking $d_0$ smaller if necessary so that $e_0$ and $v_0$ are small, the linearization $\dot{\Lambda}_{\mc{C} \coprod \mc{C}'}$ at $\xi' = \xi = 0$ restricts to  a bounded isomorphism of  $\ul{e}' \oplus \ul{e}$ onto  $\ul{e}'_* \oplus \ul{e}_*$ in the norms $\| -\|_{-\mu}^\sim$. 
Now,  let $\mc{N}_0$ denote the family of CMC necks and  set $\mc{C} * \mc{C}' : = \left\{\mc{C} \coprod \mc{C}'  \setminus (\mc{D}'_0 \coprod \mc{N}^-_{0, 0})  \right\}\coprod \mc{N}$. That is, $\mc{C} * \mc{C}'$ is obtained from $\mc{C} \coprod \mc{C}'$ be replacing the half necks $\mc{D}'_0 $  and $\mc{N}^-_{0, 0}$
 with the family of CMC necks $\mc{N}_0$. The boundary $\partial \mc{N}_0$ of $\mc{N}_0$ is  isomorphic to $\partial \left(\mc{D}'_0 \coprod \mc{N}^-_{0, 0} \right)$ and thus we can regard the pairing $*$ as a pairing on $\partial \left(\mc{C} *  \mc{C}'\right)$. Let $\tau_0$ denote the flux of the neck $\mc{N}_0$. Then it holds that 
 \[
 \left. \Lambda_{\mc{C} * \mc{C}'} \right|_{\tau_0 = 0} = \Lambda_{\mc{C} \coprod \mc{C}'}.
 \]
Thus, the linearization $\dot{\Lambda}_{\mc{C} * \mc{C}'}$ of $\Lambda_{\mc{C} * \mc{C}'}$ at $\xi + \xi' = 0$ is a bounded isomorphism  of $\ul{e} \oplus \ul{e}'$ onto  $\ul{e}_* \oplus \ul{e}'_*$ in the norms $\| -\|_{-\mu}^\sim$. By the implicit function theorem there is differentiable  map $(\tau_0, r) \mapsto (e', e)$ defined near zerp from $\ul{\tau}_0 \oplus \ul{r}$ into $\ul{e} \oplus \ul{e}'$ such that 
\begin{align}\label{Relation}
\Lambda_{\mc{C} * \mc{C}'}(\tau_0  + v + e' + e) = 0
\end{align}
 and such that $e' = e = 0$ when $r = \tau_0 = 0$. For the remainder of the proof we will consider $e$ and $e'$ to be related to $r$ and $\tau_0$ by (\ref{Relation}). By Corollary \ref{Branches!} and again taking $d_0$ smaller if necessary, there is a differentiable map $\tau_0 \rightarrow r$ such that $r = 0$ when $\tau = 0$ and such that $\mc{E}^{d}_{r + \tau_0} = 0$. We will now regard $\tau_0$ as the remaining free variable of the construction. The surface $\left(\mc{C} * \mc{C}' \right)_{\tau_0}$ is then a smoothly immersed $CMC$ surface  depending continuously on $\tau_0$ and when $\tau_0 = 0$ we have 
 \[
\left.  \left(\mc{C} * \mc{C}' \right)_{\tau_0} \right|_{\tau_0 = 0} = \mc{C} \cup \mc{C}'.
 \]
 By continuous dependence,  for small values of $\tau_0$, the ends of the surface $\mc{C} * \mc{C}'$ are asymptotically separated  and the surface $\mc{C} * \mc{C}'$ is thus asymptotically embedded. By continuous dependence   $\mc{C} * \mc{C}'$ is embedded on compact subsets of $\mb{R}^3$ for small values of $\tau_0$. Thus for $\tau_0$ sufficiently small $\mc{C} * \mc{C}'$  is globally embedded. 
\begin{proof}[Proof of Proposition \ref{BigTheorem}]
By Corollary \ref{EmbeddedChoice}  and the discussion in Section \ref{EmbeddedChoice} we can  construct completed embedded constant mean curvature surfaces with any number of ends  greater than or equal to four for any choice of finite set $\mb{F}$. Thus, in order to prove Proposition \ref{BigTheorem}, we simply need to verify that the topologies of the surfaces can be prescribed by choosing $\mb{F}$ appropriately. This is straight forward. Since our construction is invariant through the plane $\{z = 0 \}$, the set $\mb{F}$ is constrained by the fact that it includes $(i, j)$ if and only if it includes $(i, -j)$. We set $\mb{F}_{i} = \left\{ 0 \right\}$ for $i = 2, \ldots n$. For  an integer $k \geq 0$ we can either  take $\mb{F}_1 = \left\{-k, \ldots k  \right\}$  or $\mb{F}_1 = \left\{-k, \ldots k  \right\} \setminus \{ 0\}$.  In the first case the branches $\mc{B}_{0}$ and $\mc{B}_1$ are joined by $2k + 1$ necks and thus together have genus $2 k$. For $i >1$ the branch $\mc{B}_{i}$ is joined to $\mc{B}_{i - 1}$ by one neck, which contributes no genus and $\mc{B}'$ is joined to $\mc{B}_0$ by one neck, which again contribute no genus, and thus the surfaces have genus $2 k$. In the second case, the branches $\mc{B}_0$ and $\mc{B}_1$ are joined by $2k$ necks, while the topology of the rest of the construction is unchanged. By the same arguments, the surfaces have genus $2k -1$. This prescribes the genus freely and completes the proof. 
\end{proof}

\bibliographystyle{amsalpha}

 \end{document}